\documentclass[11pt,oneside]{article}
\usepackage{authblk}
\usepackage{amsthm, amsmath, amssymb, bbm, hyperref}
\usepackage{graphicx, wrapfig, caption, subcaption}
\graphicspath{ {./img/} }
\allowdisplaybreaks

\usepackage[dvipsnames]{xcolor}

\def\E{\mathbb{E}}
\def\R{\mathbb{R}}
\def\N{\mathbb{N}}
\def\Z{\mathbb{Z}}
\def\P{\mathbb{P}}
\newcommand*\dd{\mathop{}\!\mathrm{d}}

\newtheorem{theorem}{Theorem}[section]
\newtheorem{definition}{Definition}[section]

\newtheorem{remark}{Remark}[section]
\newtheorem{lemma}{Lemma}[section]

\title{Stability in stochastic hypergraph matching II: weights, batch arrivals, and continuous time}
\author{Doan Dai Nguyen, Ana Bu\v{s}i\'{c}}
\affil{Inria and DI ENS, École Normale Supérieure\\PSL University, Paris, France}

\begin{document}
	\maketitle
	\begin{abstract}
		Many real-life systems can be found as examples of stochastic matching on hypergraphs, such as production lines or assemble-to-order systems. Two common features are the number of items required may vary between matchings, and there may intermediary items which exist as a combination of other items and not of external arrivals. Both of these phenomena can be modelled by considering the weighted variant of stochastic matching.
		
		In this work, we formalise the notion of stochastic weighted matching on hypergraphs. We also allow batch arrivals, meaning multiple items of multiple classes may arrive at the same time, and in particular, the arrivals can be correlated between classes. Unlike the classical setting where items arrive at discrete time $t \in \N$, we allow arrival processes to take place in continuous time $t \in \R_{\geq 0}$.
		
		We then extend the results from the previous paper \cite{Nguyen2026} to overcome the intricacies brought up by this new setting. This allows us to derive necessary and sufficient criteria as direct generalisations of those in the unweighted setting, which depend only on the per-class arrival rates. As such, the correlation between classes bear no differences. The constructive proofs also give a maximally stable, periodic-review, size-based, arrival-rate agnostic policy.
	\end{abstract}
	
	\section{Introduction}

This article is the second in a series on necessary and sufficient criteria for the problem of stochastic matching on hypergraphs. To recall the usual setting, we are given a hypergraph $G = (V, E)$ where the vertices are the item classes and the hyperedges are the matching types. This is equipped with a distribution $\mu$ of support $V$ and a buffer of some initial state. At each time $t \in \mathbb{N}$, a vertex-$v$ is drawn i.i.d. according to $\mu$, and a class-$v$ item arrives, waiting in the buffer. A collection of items in the buffer can be matched if their corresponding classes form a hyperedge.

\subsection{Weighted matching}

It can occur, however, that multiple matchings are possible, and a matching policy is needed to determined which matchings to make, if any (even when a matching is possible, one is not obliged to realise it). The goal is to find a policy guaranteeing that no item waits indefinitely in the buffer, and the question of stability concerns determining necessary and sufficient criteria for the existence of such policy.

As we mentioned in the previous paper, this model captures many real-life applications such as assemble-to-order system \cite{Gurvich2014}, online games \cite{Veron2014}, and ride-hailing \cite{Teubner2015}, where multiple parties (for example, parts, players, rides) are required. However, in the classical setting, we assume that a matching $e$ has binary requirement. For a given vertex $v$, either a type-$e$ matching requires a class-$v$ item or it does not, which corresponds to whether $v$ is in $e$. In particular, this precludes the possibility of demanding multiple items of a given class, where the precise number of items required depending on the intended matching type. This phenomenon occurs in case of bundling in sales such as concert tickets \cite{Stremersch2002} or assembly-to-order systems.

Moreover, this also precludes the assembly-line-type system, where the small parts are combined to form intermediate bigger parts, before being eventually assembled into the final product. Given the ubiquity of such systems in industry, governing the production rate in order to ensure no bottleneck and/or waste is a crucial first step, before the question of optimal control such as storage cost (for example, production time) minimisation or profit maximisation may be systematically studied.

\subsection{Batch arrivals}

If we allow $G_{v, e}$ to take arbitrary values, it also means that a type-$e$ matching may generate more than one class-$v$ items, which then may be used for other matchings. In this light, it is thus natural to allow batch arrivals, meaning multiple items arrive at the same time.

This may seem contradictory at first, for if we ignore the probabilistic aspect of arrivals, allowing batch arrivals adds no generality, and moreover, we have single arrivals as a result of superimposing the independent Poisson processes following which the items arrive. However, we argue that superimposition also means that we are assume being able to make decisions regarding the matchings between each arrival. This is not necessarily the case for systems under heavy traffic, such as ticket venues for highly sought-after concerts, or ride-hailing during busy seasons.

\subsection{Continuous time}

It is also worth noting that in the classical setting, the arrivals and the matchings still take place at discrete time $t \in \N$, an assumption which many real-life systems do not fulfil.

A typical example is the pharmaceutical production \cite{Kashinath2025}. A production line can either take the ingredients (typically fluids) as incoming flows, which are processed in continuous manner, so-called Continuous Manufacturing. Or, traditionally, the ingredients (typically solid) are prepared as batches and passed through multiple steps with possible pauses in-between, which resembles what the stochastic hypergraph matching models. However, the choice of hybridisation is more popular \cite{Kashinath2025}, where parts of the production line are batch-based and other parts are continuous, which allows combining the benefits of both systems.

Other production lines may also be naturally described within this framework, even when they can be modelled as discrete systems \cite{Abdulmalek2006}. Once the volume of some (intermediary) products is sufficiently high, it is more beneficial to model them as a continuous flow rather than individual items.

Another scenario is when some raw materials are produced locally and whose input is regular and reliable, but other raw materials are imported from distant regions, which arrive in large batches and are subjected to transportation, tariffs, weather, and other factors which contribute to the variable delivery time. Moreover, a given raw material may itself be supplied through both channels, resulting in an input process that combines continuous flow with random batch arrivals \cite{Tkachenko2023}.

At some point in the process, a material intrinsically continuous might become discrete product \cite{Abdulmalek2006}, such as the process of packaging, so both types might also appear in the system as intermediary products.

The first problem posed by these systems is stability, meaning that we must ensure that no raw materials or intermediate products are stored indefinitely. Once this preliminary goal is obtained, one may wish to find optimal control, such as maximising profit or minimising production time. In many cases, both of these questions can be addressed by controlling the input, meaning deciding how much raw materials to order. However, this is \textit{not} always the case. An example is anaerobic co-digestion, where feedstocks are mixed to produce biogas. To maximise the yield, it is crucial that a correct mixing strategy be used \cite{Li2022}, but some of the input, such as sewage sludge, are not controllable.

These examples motivate a generalisation of stochastic matching model on hypergraphs in continuous time. In particular, the exact setting and and the question of stability without controlling the inputs constitute the topic of this work.

\subsection{Stochastic matching model on weighted hypergraphs}

To address these two shortcomings, we study a weighted model of stochastic matching on hypergraphs. The starting point is to consider the incidence matrix of the matching hypergraph, which we also denote by $G \in \{0, 1\}^{V \times E}$ without risk of misunderstanding. Each column of $G$ encodes a hyperedge $e$: one has $G_{v, e} = 1$ if and only if $v$ is in $e$, and $G_{v, e} = 0$ otherwise. In other words, one may interpret the entry $G_{v, e}$ has the number of class-$v$ items required per a type-$e$ matching.

This reinterpretation already allows $G_{v, e}$ to take arbitrary real non-negative values. We will, however, go one step further, and allow $G_{v, e}$ to take negative values as well, whilst understanding that if $G_{v, e} < 0$, then a type-$e$ matching \textit{generates} $|G_{v, e}|$ items of class $v$.

Next, to generalise the arrival process, we equip $G$ with a distribution $\mu$ on $V \times \mathbb{R}_{> 0}$ and a buffer of some initial state. Then, at time $t \in \mathbb{N}$, a pair $(v, k)$ is drawn i.i.d. according to $\mu$, and $k$ items of class $v$ arrive, waiting in the buffer. At any point, a type-$e$ matching may be realised (though, again, one is not obliged to do so) provided that the buffer contains at least $G_{v, e}$ class-$v$ items for all $v$ such that $G_{v, e} > 0$; if such a matching is realised, for each $v$ such that $G_{v, e} < 0$, there will be $|G_{v, e}|$ class-$v$ items newly generated and waiting in the buffer for other future matchings. The goal is now to find a policy guaranteeing no items, whether arrived from outside or created within the model, wait in the buffer indefinitely.

Much of these ideas can be used to generalise the arrival process to continuous time, which we will discuss at length later in this article.

\subsection{Literature review}

\subsubsection{Weighted matchings}

As we have remarked in the previous paper, before this series, little is known about stochastic matching on hypergraphs in general and stability criteria in particular. Amongst these results, that of Gupta is applicable to the non-negative integer weighted case $G_{v, e} \in \mathbb{N}$ \cite{Gupta2024}, and that of Nazari and Stolyar is applicable to general weight \cite[Remark 2]{Nazari2019}, neither of which concerns the question of stability.

In particular, Nazari and Stolyar studied the problem of reward maximisation and gave an asymptotically optimal policy with respect to a variant of long-term average reward. However, the stability of their policy is conditioned on the stability of the model $(G, \mu)$ itself \cite[Assumption 5]{Nazari2019}. On the other hand, Gupta also studied this problem for another generalisation of stochastic matching where each class might have different characteristics, but the analysis was done under General Position Condition, which is not always satisfied \cite[Section 6.4.1]{Comte2021}.

Even in case of graphs where the stability is well-characterised \cite{Comte2021, Busic2013, Mairesse2016}, this generalisation has not been properly studied. A special case of one-edge graph $e = (u, v)$ where $G_{u, e} = 1$ and $G_{v, e} \in \mathbb{N}^*$ was studied by Xiu, He, and Liu \cite{Xu1993}; recently, the condition $G_{u, e} = 1$ was relaxed to $G_{u, e} \in \mathbb{N}^*$ by Liu, Li, and Zhang \cite{Liu2024}. The authors of the latter remarked that to their best knowledge, their work was the first to address this case.

An equivalent variant of Condition \eqref{eq:inccond} was obtained independently and concurrently by Mathieu \cite{Mathieu2026} under the assumption of non-negative integer weights $G \in \N^{V \times E}$, where the solution $\lambda$ to the conservation equation $G\lambda = \mu$ is shown to be non-negative and the incidence matrix $G$ must admit a surjective sub-matrix induced by $\lambda$. In contrast, for Condition \eqref{eq:inccond}, we consider the non-negative real weights $G \in \R_{\geq 0}^{V \times E}$, and show that it suffices for $G$ to be surjective independently from $\lambda$, and there exists $\lambda$ strictly positive. The proof of sufficiency in the former case was done using the Virtual-Queue Match-the-Longest policy \cite{Mathieu2026}, a rewardless variant of the Extended Greedy Primal–Dual policy of Nazari and Stolyar \cite{Nazari2019}.

\subsubsection{Batch arrivals}

Regarding batch arrivals, amongst the works on hypergraphs, only that of Nazari and Stolyar allows multiple items to arrive at the same; more interestingly, they allow items of different classes to arrive at the same time as well. On the other hand, it is customary to consider only single-arrivals since items often arrive following Poisson processes. In case of bipartite graphs, as remarked Bu\v{s}i\'{c}, Gupta, and Mairesse \cite{Busic2013}, allowing imbalance batch arrivals in bipartite graph case leads to instability. However, it is worth noting that this remark does not necessarily apply to graph case, and eventually hypergraph case.

Note that our definition of batch arrivals does not require that all the items arriving at a given moment are of the same class, and indeed, we will allow multiple items from multiple classes to arrive. This is similar to the arrival process studied by Bu\v{s}i\'{c}, Gupta, and Mairesse \cite{Busic2013} and in contrast with the classical setting as defined by Mairesse and Moyal \cite{Mairesse2016}.

\subsubsection{Continuous time}

Regarding continuous-time arrival and matching, as far as we know, there is little literature on this topic. Indeed, stochastic matching models on graphs typically assume that items arrive one at a time, with i.i.d. types; in continuous time, this is often represented by independent Poisson arrival processes (see for instance the survey by Gardner and Righter on First-Come-First-Matched stochastic matching on bipartite graphs \cite{Gardner2020}).

Indeed, the foundational stability criteria were established under this assumption \cite{Busic2013, Mairesse2016, Comte2021}, on which further studies are based. There are other works on fluid limits \cite{Varma2020, Ozkan2020, Aveklouris2025, Aveklouris2024} and diffusion limits \cite{Xie2022, Ata2025, Gurvich2014}, these are still \textit{limits} and the underlying systems remain discrete-event systems, with individual items arriving through counting processes. To the best of our knowledge, however, stochastic matching on graphs has not been studied with primitive continuous material inputs.

A different line of work arises in the modelling of production systems with mixed continuous and discrete material flows, whereby one would use Timed Hybrid Petri Nets \cite{Julia2000, Ghaeli2005}. However, these models are generally used to analyse reachability, scheduling, conflict resolution, or performance properties of a given system, rather than to characterize policies ensuring stochastic stability.

\subsection{Contributions and organisation}
\label{subsection: contributions and organisation}

In this paper, we give some criteria characterising the stability for the two generalisations introduced above. The proof technique is similar to that in the previous paper, but with some caveat.

To start, we recall briefly the framework: we started with online assignment policies where the matching type of an item is fixed upon its arrival. This class of policies is simple enough to be analysed and characterised completely, thus leading to a stability criterion within this class \cite[Condition (1)]{Nguyen2026}. We then enlarged this class by adding reassignment to incorporate all polices, which then gives a second stability criterion with respect to all policies \cite[Condition (2)]{Nguyen2026}. We then introduced a generalisation of linear-algebraic criterion à la Comte et al. \cite[Condition (3)]{Nguyen2026}, and show that all three conditions are in fact equivalent.

However, in this new generalisation, when $G_{u, e} < 0$, each type-$e$ matching generates new class-$u$ items, and thus one cannot talk about assigning these items to $e$: they are not used by type-$e$ matchings, and if they are not used elsewhere, they will keep accumulating and make the model unstable.

Indeed, much of the proof remains similar, except that one of the key step toward the end, where we show that Condition (3) implies Condition (1) \cite[Lemma 5.2]{Nguyen2026} requires that $G_{u, e} > 0$ for all $u \in e$ and $e \in E$. Indeed, this approach suffices for non-negative-weight case, but if we wish to consider the general-weight case, which includes most importantly the assembly-line-types systems, more new ideas are needed.

Our contributions are as follow:
\begin{itemize}
	\item In Section \ref{section: preliminaries}, we introduce general-weight matching models on hypergraphs with batch arrivals in discrete time $t \in \N$. The presence of weights demands some changes, and in particular, we have a new notion of stability that generalises the one in the unweighted case.
	
	\item In Section \ref{section: online assignment policies}, Section \ref{section: reassigment and general policies}, and Section \ref{section: linear algebraic criterion and equivalence of criteria} we generalise the framework developed for the unweighted case to the case of non-negative weights.
	
	This allows us to derive three (equivalent) necessary and sufficient criteria for stability, namely Condition \eqref{eq:onlcond}, \eqref{eq:gencond}, and \eqref{eq:inccond}, which are analogous to those introduced in the previous paper \cite{Nguyen2026}. We also show that for some suitably constructed function $L_+$, the $L_+$-MaxWeight policy still has maximum stability region and is efficiently implementable.
	
	\item Then, to address the case of general weights, in Section \ref{section: generalised matchings}, we introduce a notion of generalised matching types which plays the role of matching types in the non-negative-weight case, and study some of its properties.
	
	\item However, unlike the matching types which are the (finitely many) hyperedges, there are infinitely many generalised matching types, and the framework for the case of non-negative weights, which uses the assumption that the matching hypergraph is finite, does not necessarily apply.
	
	To this end, in Section \ref{section: general weights}, we show how to reduce the stability to using only $|E|$ generalised matchings. Moreover, we show that this is optimal, and that all the results from the case of non-negative weights also apply.
	
	\item Finally, we show how to extend these results to matching models in continuous times. In Section \ref{section: continous-time matchings}, we formally define general-weight matching models on hypergraphs in continuous time $t \in \R$, which generalises the discrete-time case. Using the fact that the stability criteria only relate to the arrival process via the arrival rates, we show that they also apply to the continuous-time case. In particular, this gives us a family of maximally stable periodic-review policies.
\end{itemize}

Aside from the aforementioned, Section \ref{section: conclusion} concludes with some final remark.

A preliminary version of this work was presented in 24\textsuperscript{th} Conference of the International Federation of Operational Research Societies, 12-17 July 2026, Vienna, Austria \cite{Busic2026} and LINCS Annual Workshop, 2-3 July 2026, Palaiseau, France \cite{Nguyen2026a}.
	\section{Preliminaries}
\label{section: preliminaries}

\subsection{Matching models}

Much of the setting is similar to that in the previous paper, except that we now have to adapt state space description and matching policies to account for fractional matchings, meaning a matching where the number of items required from some class might not be an integer.

First, recall that we have a hypergraph $G = (V, E)$ on a finite set $V$ of classes, where by abuse of notation, we write $(u, e) \in G$ if $u \in e$ and $e \in E$. For each pair $(u, e)$, we assign a weight $e_u = G_{u, e}$ which is non-zero if and only if $(u, e) \in G$. 

We also equip $G$ with a distribution $\mu$ on $\R_{\geq 0}^V \setminus \{0\}$ and a buffer $\hat{Z} = (\hat{Z}_v)_{v \in V}$. For $v\in V$, et $K_v$ be the smallest sub-semigroup of $\R_{\geq 0}$ containing $\{|e_v|\mid e \in E\}$ and the support of the marginal distribution of $\mu$ on the $v$-coordinate. For convenience, let $K_V = \prod_{v \in V} K_v$, as well as
\[
	\mu_v = \int_{K_VV} k_v \dd\mu(k),
\]
and denote 
\[
	\|\mu\|_1 = \sum_{v \in V} \mu_v = \int_{K^V} \|k\|_1 \dd\mu(k),
\]
which we assume to be finite.

Before, we see $\hat{Z}_v(t)$ as the (finite) multiset of integers denoting the time of arrival of each item in the class. But since now a non-integer number of items might arrive at a time, and a matching might require a non-integer number of items from this class, this point of view is no longer appropriate. However, we can generalise this notion by seeing $\hat{Z}_v(t)$ instead as a sequence $(\hat{z}^t_{v, n})_{n \in \N} \in K^{V \times \N}$ where $\hat{z}^t_{v, n}$ is the ``number'' of items arrived at time $n$ and still in the buffer at time $t$.

Denote by $B(t) = (B_{v, n}(t))_{v \in V, n \in \N}$ the buffer state after the arrival of the batch, but before the matching policy makes decision, and by $\hat{Z}(t) = (\hat{Z}_{v, n}(t))_{v \in V, n \in \N}$ the buffer state after matchings have been made, if any.

This leads to a natural generalised definition of a matching.

\begin{definition}
	Given a hyperedge $e$, a type-$e$ matching $m$ is a sequence $(m_{v, n})_{v \in V, n \in \N}$ of numbers in $K$ such that for all $v$, we have $\sum_{n = 0}^\infty m_{v, n} = e_v$.
	
	Moreover, let $B = (b_{v, n}) \in K^{V \times \N}$ be a buffer state, then a type-$e$ matching $m$ is admissible if for all $v \in V$ and $n \in \N$, we have $m_{v, n} \leq b_{v, n}$. If we choose to realise this matching, then the resulting buffer state $C = (c_{v, n})$ is given by $c_{v, n} = b_{v, n} - m_{v, n}$.
\end{definition}

\begin{remark}
	The reason why we let $m_{v, n}$ be in $K_v$ stems from the classical setting where $K_v = \N$ and we are required to take items as a whole. Indeed, in the example of kidney exchange, we cannot give two half-kidneys to a patient.
	
	However, if $K_v = \R_{\geq 0}$, it means the items themselves are divisible, and one can talk about taking a fraction of an item. Indeed, if we now require that the ``number'' of items used for a matching to be integer, it is most likely that the model is never stable.
	
	Thus in general, we expect that the number of items used for a matching belongs to the same domain as the number of items arriving.
\end{remark}

A realisation of the model is as follow: we start with some initial buffer $\hat{Z}(0)$. At time $t$, let $A_t$ be a random variable drawn i.i.d. according to $\mu$.

Suppose the realisation of $A_t$ yields a tuple $k = (k_v)_{v \in V}$, then for each $v \in V$, $k_v$ class-$v$ items arrive, waiting in the buffer. Before any matchings are made, the buffer state is now $B(t) = (b_{v, n})$, where for all $v \in V$ and $n \in \N$, 
\[
	B_{v, n}(t) = \begin{cases}
		\hat{Z}_{v, n}(t-1) + k_v & \text{ if } n = t\\
		\hat{z}_{v, n}(t-1) & \text{ otherwise}
	\end{cases}.
\]
A matching policy $\Phi$ is then a collection of functions $(\Phi_t)_{t \in \N}$ where $\Phi_t$ determines the matchings at time $t$. Then $\hat{Z}(t)$ denotes the new buffer state.

\begin{remark}
	In this context, $\mu_v$ now is is the average number of arriving class-$v$ items, which may be assumed to be finite. This is a mild assumption and is justified from practical point of view. For some technicalities later on, we will also require that the second moment of the number of arriving items is finite, meaning
	\[
		\|\mu\|_2^2 = \int_K \|k\|^2 \dd \mu(k) < \infty.
	\]
	Both these conditions are satisfied in the classical setting, where $K = \{0, 1\}$.
\end{remark}

\begin{remark}
	Note that this means in particular that we allow the arrivals to be correlated between classes, but the arriving batch themselves still forms a Markov chain: the batch of arriving items drawn at time $t$ is independent from and identically distributed as that at time $t'$ for all $t' \neq t$.
\end{remark}

\subsection{Matching policies}

The rest of the definitions follow the previous paper. In particular, we require that $\Phi_t$ depends only on $\mathcal{H}_t$, the history up to time $t$, which is the collection of the arrivals $(A_i)_{i \leq t}$ up to time $t$ and the decisions made prior to that, given by $(\Phi_i)_{i \leq t-1}$.

This is a general description of a matching policy, but we can have relaxations such as Markovian, stationary, and size-based policies as in the classical setting.

\begin{definition}
	\hfill
	\begin{itemize}
		\item Given the buffer state $\hat{Z}(t) = (\hat{z}^t_{v, n})_{v \in V, n \in \N}$ at time $t$, let $Z(t) = (Z_v(t))_{v \in V}$ be given by
		\[
			Z_v(t) = \sum_{n = 0}^{\infty} \hat{Z}_{v, n}(t)
		\]
		then a matching policy $\Phi$ is said to be size-based if $\Phi_t$ depends only on $(Z(i))_{i \leq t-1}$ and not $(\hat{Z}(i))_{i \leq t-1}$.
		\item A matching policy $\Phi$ is said to be Markovian if $\Phi_t$ depends only on $t$, $\hat{Z}_{t-1}$ and $A_t$. Moreover, such policy is said to be stationary if $\Phi_t$ does not depend on $t$.
	\end{itemize}
\end{definition}

Many of the popular policies, such as Match-the-Longest and MaxWeight are size-based and Markovian, as they do not consider the exact time of arrival and treat all items within the same class as equal. However, there are also non-Markovian policies, such as First-Come-First-Match, where the order of arrival is considered.

\subsection{Stability}

Before, we \cite{Nguyen2026} say that $\Phi$ stabilises $(G, \mu)$ if at any given point $T \geq 0$, the time it takes to clear the buffer is finite in expectation, meaning
\[
	\mathbb{E}\left[\min\{t \geq 0 \mid \forall v, Z_v(t + T) = 0\} \mid \mathcal{H}_T\right] < \infty.
\]

However, when we allow partial matching, it can happen that the system is \textit{never} cleared. 

For example, consider a rather extreme case where the size of arriving batches takes only rational values, meaning $\mu$'s support is included in $V \times \mathbb{Q}_{> 0}$, but the number of items required from a given class for a matching is always irrational, meaning $e_u \not\in \mathbb{Q}$ for all $(u, e) \in G$. This motivates the need for a generalised notion of stability.

Even when all weights and batch sizes are integer, this can still happen: take for instance a closed network with some initial items and the matchings preserve the total number of items, but no new items arrive. Then the number of items in the buffer remains constant and strictly positive, so it is never cleared.

It turns out that we do not need the system to be cleared, but only to stabilise, meaning that the content of the buffer eventually gets ``small''. Mathematically speaking, we require that there be a compact set $C \in K^V$ such that for all $T \geq 0$, we have
\[
	\mathbb{E}\left[\min\{t \geq 0 \mid Z(t + T) \in C\} \mid \mathcal{H}_T\right] < \infty.
\]

Readers who are familiar with the theory of general state-space Markov chains and processes might recognise the similarity with the notion of positive Harris recurrence. Indeed, when if the policy is Markovian, the induced chain $Z$ is a Markov chain. Markov property gives
\[
	\mathbb{E}\left[\min\{t \geq 0 \mid Z(t + T) \in C\} \mid \mathcal{H}_T\right] = \E_{Z(T)}[\tau_C].
\]
Under the standard irreducibility assumptions, if $C$ is a petite set and the corresponding hitting-time condition satisfies the hypotheses of Meyn and Tweedie's recurrence theorems, then one obtains positive Harris recurrence of Z (see, for example, Meyn and Tweedie's \textit{Markov Chains and Stochastic Stability} \cite[Chapter 11]{Meyn2009}).

This is not a coincidence: in the classical setting, stability for Markovian policies is equivalent to the induced Markov chain being positive recurrent. We only replace ``positive recurrent'' with ``positive Harris recurrent'' as it is considered to be the correct generalisation. However, in this article, even when the policy is Markovian, we do \textit{not} require that the induced chain be positive Harris recurrent in order for the system to be stable.

	\section{Online assignment policies}
\label{section: online assignment policies}

In this Section, we assume that $e_v > 0$ for all $(v, e) \in G$.

Most of the ideas are similar to the classical setting \cite[Section 3]{Nguyen2026}, but with some caveats related to the incorporation of weights, the fact that there are infinitely many stabilising matching types, as well as the precise definition of online assignment policies.

\subsection{Online assignment policies}

We start with formally defining the notion of assignment.

\begin{definition}
	By assigning some class-$v$ items to a matching type $e \ni v$, we mean that these items then can only be matched by a type-$e$ matching. In particular, even if there are other available matching types using class $v$, these items will still not be matched.
\end{definition}

Note that by ``some'' items, we mean that the number of items can be non-integer. Next, we define the notion of assignment rates.

\begin{definition}
	Given an arrival rate $\mu$, we call a tuple $\nu = (\nu_{v, e})_{(v, e) \in G} \in \R_{\geq 0}^G$ an assignment rate if for all $v \in V$, we have $\sum_{e \ni v} \nu_{v, e} = \mu_v$.
\end{definition}

An online assignment policy $\Phi$ is given as follow: at time $t$, an assignment rate $\nu$ is chosen. Then, a tuple $A_t = (k_v)$ is drawn i.i.d. from $\mu$, and for each $v\in V$, $k_v$ class-$v$ items arrive. Then, for each $v$, we then assign \textit{all} $k_v$ items to a matching type $e$ with probability $\frac{\nu_{v, e}}{\mu_v}$; the definition of assignment rate ensures that this is a proper probability distribution. In case $\mu_v = 0$, we have a divide-by-0 situation, but this is not a danger since by definition, class-$u$ items (almost surely) never arrives, so the policy can have its behaviour undefined.

Once assigned, the items within the buffer for the matching type $e$ form matchings together in first-come-first-matched fashion. Once all the possible matchings are formed, we denote by $X_{v, e}(t)$ to be the number of class-$v$ items assigned to type-$e$ matching which remain in the buffer. This forms a process $X$ where $X(t)$ is an element of $\R_{\geq 0}^G$. Note that this means that our definition of online assignment policies requires it to be size-based.

Similar to the stability with respect to general policies, we say that an online assignment policy stabilises the model $(G, \mu)$ if there exists a compact set $C \subset \R_{\geq 0}^G$ such that for all $T > 0$, we have
\[
	\E[\min\{t \geq 0 \mid X(T + t) \in C\} \mid \mathcal{H}_T] < \infty, 
\]
meaning that at any given point $T$, the expected time before the enlarged buffer returns to the set $C$ is finite.

Some remarks are in order, notably regarding the definition of online assignment policies and of batch arrivals.

\begin{remark}
	In case $\{1\} \subsetneq K \subseteq \N$, meaning that the items may arrive in batches but are indivisible, one may ask why we choose to define online assignment policies to be assigning the whole batch to a matching type instead of assigning individual items.
	
	The first reason stems from the practical point of view. In applications, there are cases when the incoming traffic exceeds the capacity of the system and overwhelms the decision maker. If the traffic is due to the number of classes, the problem lies purely in the limited computational capacity. However, if the hypergraph is still relatively small, the traffic is due to the numerous arrivals at some classes. In this case, we may speed up the matching if we allow assignment in batches.
	
	The second reason stems from the technical point of view. This definition simplifies and generalises to the case where the items are infinitely divisible, for instance when $K = \R$. One can certainly still assignment parts of the arriving item batch to (finitely many) different matching types, but as we will see, our rather restrictive definition in fact incurs no cost of generality.
\end{remark}

\subsection{Necessary condition for stability}

We first revisit the informal argument that leads to our first stability criterion for unweighted case \cite[Condition (1)]{Nguyen2026}. Consider the support $U = \{(v, e) \mid X_{v, e}(t) > 0\}$ of the buffer state $X(t)$ at time $t$, there are essentially two forces at play that opposes each other:
\begin{itemize}
	\item for each pair $(v, e) \in U$, an arriving class-$v$ item which is assigned to $e$ increases the number of items by $1$; 
	\item but at the same time, each arriving item which is assigned to $e$ also brings $\frac{1}{|e|}$ matching, hence decreases the number of items by $\frac{1}{|e|}$.
\end{itemize}
The stability criterion then roughly says the the latter outweighs the former. This informal intuition still holds for weighted cases.

Consider the buffer state $X(t)$ at time $t$. Within one stabilising matching type $e \in E$, for each $v \in e$, there are $X_{v, e}(t)$ items. All the items form type-$e$ matchings in first-come-first-matched fashion; there are exactly $\min_{v \in e} X_{v, e}(t)$ such matching formed. As an online assignment policy makes all of them, we can assume that 
\[
	\min_{v \in e} X_{v, e}(t) - e_v < 0,
\]
meaning that the content of the buffer is not sufficient to make a type-$e$ matching.

Then we consider the set $U = \{(v, e) \mid X_{v, e} \geq e_v \} \subset G$ the support of the buffer.

\begin{definition}
	\label{definition: support}
	A set $U \subset G$ is called a support if for all $e \in E$, the $e$-slice of $U$, given by $U_e = \{v \mid (v, e) \in U\}$, cannot be $e$, or in other words, there exists some vertex $v$ such that $(v, e) \in G \setminus U$.
\end{definition}

From the definition, a support of any buffer state is a support in the above sense, and conversely, to every support, we can define a buffer state admitting such set as its support in the usual sense.

Now suppose that $(G, \mu)$ is stabilised by some online assignment policy $\Phi$, which chooses at time $t$ an assignment rate $\nu$. On average, there will be $\nu_{v, e}$ class-$v$ items arriving and being assigned to $e$, which will require $\frac{\nu_{v, e}}{e_v}$ type-$e$ matchings. On the other hand, on average, there will be $\nu_e = \sum_{v \in e} \nu_{v, e}$ arriving items being assigned to $e$. A type-$e$ matching requires $|e| = \sum_{u \in e} e_u$ items, thus these arriving items bring $\frac{\nu_e}{|e|}$ matchings.

Thus, heuristically, the number of required matchings after the items arrived and assigned increases by
\[
	\frac{\nu_{v, e}}{e_v} - \frac{\nu_e}{|e|},
\]
Consider a given matching type $e$ and a class $v \in e$, if we have in abundance class-$v$ items assigned to $e$, meaning $X_{v, e}(t) \geq e_v$ then we may expect the quantity above to be negative.

This leads us to the following condition.
\begin{equation}
	\label{eq:onlcond}
	\textrm{\parbox{.8\textwidth}{For all supports $U$, there exists an assignment rate $\nu^U$ such that for all $(v, e) \in U$,
	\[
		\frac{\nu^U_{v, e}}{e_v} < \frac{\nu^U_e}{|e|}.
	\]
	}}
\end{equation}
Notice that this is more elaborate than \texttt{NCOND} \cite{Mairesse2016} and our stability criterion in unweighted case \cite[Condition (1)]{Nguyen2026}. Nevertheless, we will see that it is still a necessary condition.

\begin{lemma}
	\label{lemma: necessity of onlcond}
	If $(G, \mu)$ is stabilisable by an online assignment policy, then Condition \eqref{eq:onlcond} is satisfied.
\end{lemma}
\begin{proof}
	We need to modify the construction of the process $Y$ in the proof of necessity for Condition (1) in the unweighted case \cite[Lemma 3.1]{Nguyen2026} in order to account for batch arrival and weighted matching. Moreover, as the buffer might not be cleared and the condition is slightly different, the proof must also be adapted, but the central idea remains the same.
	
	First, we fix an online assignment policy $\Phi$ stabilising $(G, \mu)$. Let $C$ be a compact set for which the stability condition is satisfied by definition.
	
	Now we start a realisation of the system. Let it assign the initial buffer arbitrarily, and let $Y_{v, e} = X_{v, e}(0)$ for all $(v, e) \in G$. Suppose for each $v\in V$, $k_v$ class-$v$ items arrive and the policy decides to assign it to a matching type $e$. We then define
	\[
		Y_{v, e}(t + 1) = Y_{v, e}(t) + \frac{k_v}{e_u} - \frac{\sum_{u \in e, k_u\text{ class-$u$ items are assigned to } e} k_u}{|e|}.
	\]
	By induction, we have that
	\[
		Y_{v, e}(t) = \frac{1}{e_v} X_{v, e}(t) - \frac{1}{|e|} \sum_{e \ni u} X_{u, e}(t).
	\]
	
	Now assume for the sake of contradiction that Condition \eqref{eq:onlcond} is not satisfied, meaning that there exists a support $U \subset G$ such that for all assignment rate $\nu$, there exists either some $(v, e) \in U$ such that 
	\[
		\frac{\nu_{v, e}}{e_v} \geq \frac{\nu_e}{|e|}.
	\]
	
	Given an assignment rate $\nu$, we define a map $\Delta : \R^G_{\geq 0} \to \R^U$ by
	\[
		(\Delta \nu)_{v, e} = \frac{\nu_{v, e}}{e_v} - \frac{\nu_e}{|e|}
	\]
	which is a linear map on $\nu$. The set of assignment rates $\nu$ is a convex polytope $\mathcal{P}$ defined by the following constraints:
	\begin{itemize}
		\item for all $(v, e) \in G$, $\nu_{v, e} \geq 0$, and
		\item for all $v \in V$, $\sum_{e \ni v} \nu_{v, e} = \mu_v$.
	\end{itemize}
	
	Since $\Delta$ is a linear map, the image $\Delta(\mathcal{P})$ is also a convex polytope in $\R^U$. As \eqref{eq:onlcond} is not satisfied, every point $x \in \Delta(\mathcal{P})$ has some non-negative coordinate.
	
	In other words, $\Delta(\mathcal{P})$ and $\R^U_{\leq 0}$ have disjoint interior. As both sets are compact and convex, hyperplane separation theorem implies that there exists a hyperplane $H$ separating the two sets. Up to a translation, $H$ can be made to be tangent to $\R^U_{\leq 0}$, but as it is a cone, $H$ will pass by the origin.
	
	Let $w$ be a normal vector which is non-zero since $H$ is non-degenerate; we normalise it to have $\|w\|_1 = 1$. As $H$ separates $\Delta(\mathcal{P})$ and $\R^U_{\leq 0}$, we can even choose $w$ such that
	\begin{itemize}
		\item for all $x \in \Delta(\mathcal{P})$, we have $\langle x, w \rangle \geq 0$, and
		\item for all $x \in \R^U_{\leq 0}$, we have $\langle x, w \rangle \leq 0$.
	\end{itemize}
	
	The last condition is equivalent to saying that for all $x \in \R^U_{\leq 0}$, we have $\langle x, w \rangle \leq 0$: indeed, it suffices to replace $x$ by $x + \varepsilon \mathbbm{1}$. By replacing $x$ with $-x$, this is also equivalent to saying that for all $x \in \R^U_{\geq 0}$, we have $\langle x, w \rangle \geq 0$, meaning that $w$ lies in the dual cone of $\R^U_{\geq 0}$. But since this is a self-dual cone, we have $w \in \R^U_{\geq 0}$, or $w$ has only non-negative coordinates.
	
	Now we define a function $L_-$: for a given state $y \in \R_{\geq 0}^G$, we have 
	\[
		L_- (y) = \sum_{(v, e) \in U} w_{v, e} y_{v, e},
	\].
	
	By construction of $Y$, for $(v, e) \in G$, if $(v, e) \in U$ and $k_v$ class-$v$ items arrive and are assigned to the hyperedge $e$, its contribution to the difference $L_- (Y(t+1)) - L_- (Y(t))$ is given by 
	\[
		k_v\left[w_{v, e}\left(\frac{1}{e_u} - \frac{1}{|e|}\right) - \sum_{(u, e) \in U, u \neq v} \frac{w_{u, e}}{|e|}\right] = k_v\left[w_{v, e} - \frac{\sum_{(u, e) \in U} w_{u, e}}{|e|}\right],
	\]
	and $0$ otherwise. Therefore, conditioning on $\hat{X}(t)$ and the behaviour of $\Phi$ at time $t$ denoted by an assignment rate $\nu$, by the first property of the separating plane $H$, we have
	\begin{align*}
		& \E\left[L_-(Y(t+1)) - L_-(Y(t))\right] \\
		= & \sum_{(v, e) \in U} \nu_{v, e} \left(w_{v, e} - \frac{\sum_{(u, e) \in U} w_{u, e}}{|e|}\right) \\
		= & \sum_{(v, e) \in U} \nu_{v, e} w_{v, e} - \sum_{(v, e) \in U} \nu_{v, e} \frac{\sum_{(u, e) \in U} w_{u, e}}{|e|} \\
		= & \sum_{(v, e) \in U} \nu_{v, e} w_{v, e} - \sum_{(v, e) \in U} w_{v, e} \frac{\sum_{(u, e) \in U} \nu_{u, e}}{|e|} \\
		= & \sum_{(v, e) \in U} w_{v, e} \left(\nu_{v, e} - \frac{\sum_{(u, e) \in U} \nu_{u, e}}{|e|} \right) \\
		= & \langle w, \Delta \nu \rangle \geq 0,
	\end{align*}
	which can also be written as $\E\left[L_-(Y(t)) - L_-(Y(t+1))\right] \geq 0$. By induction, we obtain $L_-(Y(0)) \leq \E\left[L_-(Y(t))\right]$.
	
	Denote $e_G = \min_{(v, e) \in G} e_v > 0$, we have
	\begin{align*}
		e_G \|Y(t)\| 
		& \leq \sum_{(v, e) \in G} e_v |Y_{v, e}(t)| = \sum_{(v, e) \in G} \left| X_{v, e}(t) - \frac{e_v}{|e|} \sum_{e \ni u} X_{u, e}(t) \right| \\
		& \leq \sum_{(v, e) \in G} \left(|X_{v, e}(t)| + \frac{e_v}{|e|} \sum_{e \ni u} |X_{u, e}(t)|\right) = 2\|X(t)\|_1.
	\end{align*}
	and in particular, for
	\[
		\tau = \min\{t \geq 0 \mid X(t) \in C\} < \infty, 
	\]
	we have $\|Y(\tau)\| \leq \frac{2}{e_G} \|X(\tau)\|_1 \leq \frac{2}{e_G}\max_{x \in C} \|x\|_1$. Combining with the bound above for $t = \tau$, we have $L_-(Y(0)) \leq \frac{2}{e_G}\max_{x \in C} \|x\|_1$.
	
	However, this last inequality must also hold for all initial state $Y(0)$, and whilst the right-hand side is constant, we can choose $Y(0)$ to make $L_-(Y(0))$ arbitrarily large regardless of how the policy assigns the initial buffer. This gives a contradiction as desired.
\end{proof}

\subsection{Sufficiency of Condition \eqref{eq:onlcond}}

As we may expected from the unweighted case, Condition \eqref{eq:onlcond} is also sufficient for stability within the class of online assignment policies. Much of the proof from the unweighted case carries over; this is because in fact we have not really used the fact that only one item arrives at a time, and we only work with the average number of arriving items. However, there is one step in the proof where we will need the fact that $\mu_v^2$ is finite for all $v \in V$.

To outline the proof, we introduce a Lyapunov function $L_+$, which is the weighted version of the function we introduced in the previous paper \cite[Section 3.3]{Nguyen2026}, as well as a deterministic policy of type MaxWeight which aims at minimising $L_+$ myopically.

Assume that $(G, \mu)$ is stabilisable, we know that \eqref{eq:onlcond} is satisfied. Using this condition, we construct a randomised policy depending on $\mu$ which achieves negative drift with respect to $L_+$ outside a finite region; this finite region then plays the role of the compact set in the stability definition. Since $L_+$-MaxWeight minimises the drift myopically, it too stabilises $(G, \mu)$.

In the unweighted case, the function we used was
\[
	\frac{1}{2} \sum_{e \in E} \sum_{\{u, v\} \in e^2} \left(x_{u, e} - x_{v, e}\right)^2.
\]
In particular, the term $\left(x_{u, e} - x_{v, e}\right)^2$ measures the pairwise imbalance. Ideally, we would have the same number of class-$u$ items as class-$v$ items, all of which then form perfect matchings; the sum $\sum_{{u, v} \in e^2} \left(x_{u, e} - x_{v, e}\right)^2$ then measures how far away we are from this ideal state.

The idea for weighted case is similar, except that now the number of items required from each class for a matching might not be the same: if a matching type $e$ requires $e_u$ class-$u$ items and $e_v$ class-$v$ items, ideally we would like $\frac{x_{u, e}}{e_u} = \frac{x_{v, e}}{e_v}$. The difference between these two quantities is the number of missing matchings. Naturally, this motivates the following Lyapunov function.
\[
	\frac{1}{2} \sum_{e \in E} \sum_{\{u, v\} \in e^2} \left(\frac{x_{u, e}}{e_u} - \frac{x_{v, e}}{e_v}\right)^2.
\]

Due to some technicalities in the proof, we will instead use the following function
\[
	L_+(x) = \frac{1}{2} \sum_{e \in E} \frac{1}{|e|}\sum_{\{u, v\} \in e^2} e_u e_v \left(\frac{x_{u, e}}{e_u} - \frac{x_{v, e}}{e_v}\right)^2.
\]
And indeed, the rest of the analysis carries over.

\begin{lemma}
	\label{lemma: sufficiency of onlcond}
	If Condition \eqref{eq:onlcond} is satisfied, $(G, \mu)$ is stabilisable by a randomised Markovian size-based online assignment policy and by (the online assignment version of) $L_+$-MaxWeight.
\end{lemma}
\begin{proof}
	First, for clarity's sake, we define our policy $\Phi$. It assigns the initial buffer state arbitrarily. Then, at time $t$, it considers the support $U_t$ of the buffer state $X(t)$, and takes an assignment rate $\nu_t$ satisfying Condition \eqref{eq:onlcond}, and proceeds as usual.
	
	To show that $\Phi$ is maximally stabilising, consider the buffer state $X(t) = x \in \R^G$ at time $t$. For given pairs $(u, e), (v, e) \in G$, assume that $k_{u, e}$ class-$u$ and $k_{v, e}$ class-$v$ item arrives and are assigned to the matching type $e$. Denote $e_G = \min_{(v, e) \in G} e_v$, we have that the change $\Delta \left[\left(\frac{x_{u, e}}{e_u} - \frac{x_{v, e}}{e_v}\right)^2\right]$ of the term $\left(\frac{x_{u, e}}{e_u} - \frac{x_{v, e}}{e_v}\right)^2$ is given by
	\begin{align*}
		& \Delta \left[\left(\frac{x_{u, e}}{e_u} - \frac{x_{v, e}}{e_v}\right)^2\right]\\
		= &2\left(\frac{k_{u, e}}{e_u} - \frac{k_{v, e}}{e_v}\right)\left(\frac{x_{u, e}}{e_u} - \frac{x_{v, e}}{e_v}\right) + \left(\frac{k_{u, e}}{e_u} - \frac{k_{v, e}}{e_v}\right)^2 \\
		\leq & \frac{k_{u, e}^2}{e_u^2} + \frac{k_{v, e}^2}{e_v^2} + 2\left(\frac{k_{u, e}}{e_u} - \frac{k_{v, e}}{e_v}\right)\left(\frac{x_{u, e}}{e_u} - \frac{x_{v, e}}{e_v}\right).
	\end{align*}
	Summing over all such terms, we have
	\[
	\begin{aligned}
		& \Delta L_+(X(t))\\
		= & \frac{1}{2} \sum_{e \in E} \frac{1}{|e|}\sum_{\{u, v\} \in e^2} e_u e_v \left[\frac{k_{u, e}^2}{e_u^2} + \frac{k_{v, e}^2}{e_v^2} + 2\left(\frac{k_{u, e}}{e_u} - \frac{k_{v, e}}{e_v}\right)\left(\frac{x_{u, e}}{e_u} - \frac{x_{v, e}}{e_v}\right)\right] \\
		\leq & \frac{1}{2} \sum_{e \in E} \sum_{\{u, v\} \in e^2} k_{u, e}^2 + k_{v, e}^2 \\
		& + \sum_{e \in E} \frac{1}{|e|}\sum_{\{u, v\} \in e^2} \left(\frac{e_v}{e_u} k_{u, e} x_{u, e} + \frac{e_u}{e_v} k_{v, e} x_{v, e} - k_{u, e} x_{v, e} - k_{v, e} x_{u, e}\right)\\
		\leq & \frac{|V|}{2} \sum_{e \in E} \sum_{u \in e} k_{u, e}^2 \\
		& + \frac{1}{2} \sum_{e \in E} \frac{1}{|e|} \sum_{u \in e} \sum_{v \in e} \left(\frac{e_v}{e_u} k_{u, e} x_{u, e} + \frac{e_u}{e_v} k_{v, e} x_{v, e} - k_{u, e} x_{v, e} - k_{v, e} x_{u, e}\right).
	\end{aligned}
	\]
	Let us deal with the big sum term-by-term. The first term (and by symmetry, the second term) reads
	\[
		\sum_{u \in e} \sum_{v \in e} \frac{e_v}{e_u} k_{u, e} x_{u, e} = \sum_{u \in e} \frac{1}{e_u} k_{u, e} x_{u, e} \sum_{v \in e} |e| =  \sum_{u \in e} \frac{|e|}{e_u} k_{u, e} x_{u, e}.
	\]
	The third term (and by symmetry, the fourth term) reads
	\[
		\sum_{u \in e} \sum_{v \in e} k_{u, e} x_{v, e} = \sum_{u \in e} k_{u, e} \sum_{v \in e} x_{v, e} = \sum_{u \in e} k_{u, e} x_e,
	\]
	where we denote $x_e = \sum_{v \in e} x_{v, e}$. Thus in summary, we have
	\[
	\begin{aligned}
		\Delta L_+(X(t))
		& \leq \frac{|V|}{2} \sum_{e \in E} \sum_{u \in e} k_{u, e}^2 + \sum_{e \in E} \frac{1}{|e|} \left[\sum_{u \in e} \frac{|e|}{e_u} k_{u, e} x_{u, e} - \sum_{u \in e} k_{u, e} x_e\right] \\
		& \leq \frac{|V|}{2} \sum_{e \in E} k_u^2 + \sum_{(u, e) \in G} k_{u, e} \left(\frac{x_{u, e}}{e_u} - \frac{x_e}{|e|}\right),
	\end{aligned}
	\]
	where we denote $k_u = \sum_{e \ni u} k_{u, e}$ the total number of arriving class-$u$ items.
	
	Now for a given assignment rate $\nu$ and taking expectation over the arrivals, we have
	\begin{align*}
		\E[\Delta L_+ \mid X(T) = x]
		& \leq \frac{|V|}{2} \|\mu\|_2^2 + \sum_{(u, e) \in G} \nu_{u, e} \left(\frac{x_{u, e}}{e_u} - \frac{x_e}{|e|}\right).
	\end{align*}
	
	The second term can be rewritten as follow:
	\begin{align*}
		& \sum_{(u, e) \in G} \nu_{u, e} \left(\frac{x_{u, e}}{e_u} - \frac{x_m}{|e|}\right) = \sum_{e \in E} \left[\sum_{u \in e} \nu_{u, e} \left(\frac{x_{u, e}}{e_u} - \frac{x_m}{|e|}\right)\right] \\
		= & \sum_{e \in E} \left[\sum_{u \in e} \nu_{u, e} \frac{x_{u, e}}{e_u} - \left(\sum_{u \in e} \nu_{u, e}\right)\frac{x_m}{|e|}\right] \\
		= & \sum_{e \in E} \left[\sum_{u \in e} \nu_{u, e} \frac{x_{u, e}}{e_u} - \nu_e \frac{x_m}{|e|}\right] \\
		= & \sum_{e \in E} \left[\sum_{u \in e} \nu_{u, e} \frac{x_{u, e}}{e_u} - \left(\sum_{u \in e} x_{u, e}\right)\frac{\nu_e}{|e|}\right] \\
		= & \sum_{e \in E} \left[\sum_{u \in e} x_{u, e} \left(\frac{\nu_{u, e}}{e_u} - \frac{\nu_e}{|e|}\right)\right] \\
		= & \sum_{(u, e) \in G} x_{u, e} \left(\frac{\nu_{u, e}}{e_u} - \frac{\nu_e}{|e|}\right) \\
		= & \sum_{(u, e) \in G \setminus U} x_{u, e} \left(\frac{\nu_{u, e}}{e_u} - \frac{\nu_e}{|e|}\right) + \sum_{(u, e) \in U} x_{u, e} \left(\frac{\nu_{u, e}}{e_u} - \frac{\nu_e}{|e|}\right) \\
	\end{align*}
	where $U$ is the support of $x$.

	The first term can be bounded trivially, as by definition, we have $x_{u, e} < e_u$ for all $(u, e) \not\in U$, 
	\begin{align*}
		\sum_{(u, e) \in G \setminus U} x_{u, e} \left(\frac{\nu_{u, e}}{e_u} - \frac{\nu_e}{|e|}\right)  
		& \leq \sum_{(u, e) \in G \setminus U} x_{u, e} \left(\frac{\nu_{u, e}}{e_u} + \frac{\nu_e}{|e|}\right) \\
		& \leq \sum_{(u, e) \in G \setminus U} e_u \left(\frac{\nu_{u, e}}{e_u} + \frac{\nu_e}{|e|}\right) \\
		& \leq \sum_{e \in E} \sum_{u \in e} \nu_{u, e} + \frac{e_u}{|e|} \nu_e = 2\|\mu\|_1 \\
	\end{align*}
	
	For the second term, by definition of support, we have $\frac{\nu_{u, e}}{e_u} - \frac{\nu_e}{|e|}$ to be negative for all $(u, e) \in U$. In particular, since $G_U$ is finite, there exists $\varepsilon > 0$ such that $\frac{\nu_{u, e}}{e_u} - \frac{\nu_e}{|e|} \leq -\varepsilon$ for all $(u, e) \in U$. Moreover, since $G$ is finite, there are finitely many supports $U$, so in particular, we can even choose $\varepsilon$ so that this holds \textit{uniformly} for all $U$.
	
	As such, we have
	\[
		x_{u, e} \left(\frac{\nu_{u, e}}{e_u} - \frac{\nu_e}{|e|}\right) \leq -\varepsilon x_{u, e},
	\]
	implying
	\begin{align*}
		\E[\Delta L_+ \mid X(T) = x]
		& \leq \frac{|V|}{2} \|\mu\|_2^2 + 2\|\mu\|_1 - \varepsilon \sum_{(u, e) \in U} x_{u, e} \\
		& \leq \frac{|V|}{2} \|\mu\|_2^2 + 2\|\mu\|_1 + \varepsilon\sum_{(u, e) \in G} e_u - \varepsilon \sum_{(u, e) \in G} x_{u, e}\\
		& \leq C - \varepsilon \|x\|_1 \\
	\end{align*}
	where for the penultimate inequality, we use the fact that $x_{u, e} < e_u$ for all $(u, e) \not\in U$. For the last inequality, we denote $C = \frac{|V|}{2} \|\mu\|_2^2 + 2\|\mu\|_1 + \varepsilon\|G\|_1$, which is absolute constant depending only $G$ and $\mu$ (note that our choice of $\varepsilon$ does not depend on $U$).
	
	Thus, for an arbitrarily chosen $\delta > 0$ and
	\[
		C_\delta = \left\{x \Biggr\vert \|x\|_1 \geq \frac{C + \delta}{\varepsilon}\right\},
	\]
	we have that $\E[\Delta L_+ \mid X(t) = x] \leq -\delta$ for all $x \not\in C_\delta$. In particular, as $\Phi$ is Markovian, the chain $X(t)$ is a Markov chain, so the negative drift implies that
	\[
		\E[\min\{t \in \N \mid X(T + t) \in C_\delta\} \mid X(T) = x] \leq \frac{L_+(x)}{\delta},
	\]
	which is finite, as desired.
	
	Finally, since $L_+$-MaxWeight optimises the drift, we also get as a corollary that $L_+$-MaxWeight, in the online assignment setting, also stabilises $(G, \mu)$.
\end{proof}

Combining with Lemma \ref{lemma: necessity of onlcond}, we have the necessity and sufficiency of Condition \eqref{eq:onlcond}.

\begin{theorem}
	\label{theorem: necessity and sufficiency of onlcond}
	Let $G \in \R_{\geq 0}^{V\times E}$. Condition \eqref{eq:onlcond} is necessary and sufficient for stability within the class of online assignment policies: $(G, \mu)$ is stabilisable by an online assignment policy if and only if Condition \eqref{eq:onlcond} is satisfied.
	
	When this is the case, there are stabilising online assignment policies which are Markovian, randomised, and size-based. Moreover, $(G, \mu)$ is also stabilised by the online assignment variant of $L_+$-MaxWeight.
\end{theorem}
	\section{Reassignment and general policies}
\label{section: reassigment and general policies}

In this Section, we continue assuming that $e_v > 0$ for all $(v, e) \in G$.

As in the unweighted case \cite[Section 4]{Nguyen2026}, to accommodate general policies, we enlarge the online assignment framework by adding the notion of reassignment. The differences in the setting of the weighted case requires that the definition of reassignment rate be generalised.

\begin{definition}
	Given a buffer state $x$, a reassignment rate $\gamma$ is a distribution on $K_x = \{k = \left(k_{v, e, e'}\right)_{v \in e, e'} \mid \sum_{e'} k_{v, e, e'} = x_{v, e} \wedge k_{v, e, e'} \geq 0\}$, and we say that this reassignment rate $\gamma$ is adapted to $x$.
\end{definition}

\textit{Any} policy $\Phi$ can have its behaviour at time $t$ described by a pair $(\nu, \gamma)$ of an assignment rate $\nu$ and a reassignment rate $\gamma$ adapted to the current buffer state $X(t)$, as follow.
\begin{itemize}
	\item At time $t$, we first assign the newly arriving items. If some class-$v$ items arrive, we assign them all to the matching type $m$ with probability $\frac{\nu_{v, e}}{\mu_e}$.
	\item Then, we draw a sequence $k = (k_{v, e, e'})_{v \in e, e'}$ from $\gamma$, and for each $v \in e, e'$, we reassign $k_{v, e, e'}$ class-$v$ items from the matching type $m$ to $e'$.
\end{itemize}
The definition of reassignment rates ensures that this latter operation is valid, and the resulting buffer state still admits a support (in the sense of Definition \ref{definition: support}) as its support (in the usual sense).

To derive a necessary criterion for stability in this new context, we go back to the informal argument in the previous Subsection. Let
\[
	\gamma_{v, e, e'} = \int_{K_x} k_{v, e, e'} \dd \gamma(k),
\]
this quantity denotes the average number of class-$v$ items reassigned from $m$ to $e'$, thus the change in the number of class-$v$ items assigned to $m$ is given by
\[
	\Gamma_{v, e} = \sum_{e' \ni v, e' \neq m} \gamma_{v, e', m} - \gamma_{v, e, e'}.
\]
At the same time, these newly assigned items bring with them $\frac{mu_u}{|e|}\Gamma_{v, e}$ matchings. A similar reasoning suggests that if we choose to reassign, the quantity
\[
	\frac{\Gamma_{v, e}}{e_u} - \frac{\Gamma_e}{|e|}
\]
should be negative for all $(v, e) \in U$, where $\Gamma_e = \sum_{v \in m} \Gamma_{v, e}$. Thus, by combining both terms of assignment and reassignment, we expect the negative sign for the quantity
\[
	\frac{\nu_{v, e}}{e_u} - \frac{\nu_e}{|e|} + \frac{\Gamma_{v, e}}{e_u} - \frac{\Gamma_e}{|e|}.
\]
Now we can formulate our necessary condition for stability, this time within the class of all policies.
\begin{equation}
	\label{eq:gencond}
	\textrm{\parbox{.8\textwidth}{For all supports $U$, there exists a pair $(\nu^U, \gamma^U)$ such that for all $(v, e) \in U$,
	\[
		\frac{1}{e_u} \left(\nu^U_{v, e} + \Gamma^U_{v, e}\right) < \frac{1}{|e|}\left(\nu^U_e + \Gamma^U_e\right).
	\]
	}}
\end{equation}

\begin{remark}
	Readers of the previous article \cite{Nguyen2026} may notice that unlike the unweighted case, our Condition \eqref{eq:gencond} does not have the scalar $\alpha$. This is because in the unweighted case, each reassignment only move one item, and we may repeat this step an arbitrary number of time. However, as we do not know the smallest unit of item in the weighted case, we have chosen to combine all the reassignments into one step. 
\end{remark}

Much of the proof of necessity remains the same.
\begin{lemma}
	\label{lemma: necessity of gencond}
	If $(G, \mu)$ is stabilisable, then Condition \eqref{eq:gencond} is satisfied.
\end{lemma}
\begin{proof}
	We only have to adapt the construction of the process $Y$ from the proof of Lemma \ref{lemma: necessity of onlcond}. In particular, let $Y'(t)$ be the process $Y$ at time $t$ after the assignment. When there is a reassignment $k$, writing $k_{v, e} = \sum_{e' \ni v, e' \neq m} k_{v, e', m} - k_{v, e, e'}$ and $k_e = \sum_{v \in m} k_{v, e}$, we define
	\[
		Y_{v, e}(t+1) = Y'_{v, e}(t) - \frac{k_{v, e}}{e_v} + \frac{k_e}{|e|}.
	\]
	This modification is precisely to retain the formula
	\[
		Y_{v, e}(t) = \frac{1}{e_v} X_{v, e}(t) - \frac{1}{|e|} \sum_{e \ni u} X_{u, e}(t).
	\]
	The rest of the proof goes analogously to that of Lemma \ref{lemma: necessity of onlcond}.
\end{proof}

Similarly, as with the unweighted case \cite[Lemma 4.2]{Nguyen2026}, this condition is also sufficient for stability. The proof goes analogously to that for online assignment polices, which we omit for the sake of brevity. 

\begin{lemma}
	\label{lemma: sufficiency of gencond}
	If Condition \eqref{eq:gencond} is satisfied, $(G, \mu)$ is stabilisable by a randomised Markovian size-based policy and by $L_+$-MaxWeight.
\end{lemma}

Combining Lemma \ref{lemma: necessity of gencond} and Lemma \ref{lemma: sufficiency of gencond}, we have a characterisation of stability.

\begin{theorem}
	\label{theorem: necessity and sufficiency of gencond}
	Let $G \in \R_{\geq 0}^{V\times E}$. Condition \eqref{eq:gencond} is necessary and sufficient for stability: $(G, \mu)$ is stabilisable if and only if Condition \eqref{eq:gencond} is satisfied.
	
	When this is the case, there are stabilising policies which are Markovian, randomised, and size-based. Moreover, $(G, \mu)$ is also stabilised by $L_+$-MaxWeight.
\end{theorem}
	\section{Linear algebraic criterion and equivalence of criteria}
\label{section: linear algebraic criterion and equivalence of criteria}

In this Section, as is in the previous two Sections, we also assume that $e_v > 0$ for all $(v, e) \in G$.

\subsection{Linear algebraic criterion}
As we would expect from the unweighted case, Condition \eqref{eq:onlcond} and  \eqref{eq:gencond} are equivalent to another necessity condition expressed in terms of the existence of a positive solution to the conservation equation and the surjectivity of some matrix.

The expected linear algebraic necessity condition is thus given below. Much of the proof of necessity is directly from the unweighted case \cite[Lemma 5.1]{Nguyen2026}.

\begin{equation}
	\label{eq:inccond}
	\textrm{\parbox{.8\textwidth}{Let $\mu = (\mu_v)_{v \in V} \in \R_{\geq 0}^V$, then $G \lambda = \mu$ has a positive solution $\lambda \in \mathbb{R}_{> 0}^E$ and $G$ is surjective.
	}}
\end{equation}

Note that we are abusing notation, since in the pair $(G, \mu)$, $\mu$ is a \textit{distribution} on $\R_{> 0}^V$ and not the vector $(\mu_v)_{v \in V}$. However, as stability concerns only the latter and not the former, we identify $\mu$ with $(\mu_v)_{v \in V}$.

\begin{lemma}
	\label{lemma: necessity of incccond}
	Condition \eqref{eq:gencond} implies Condition \eqref{eq:inccond}.
\end{lemma}
\begin{proof}
	First, to show the surjectivity of $G_M$, we note that for a given hypergraph $G$, the set $\mu$ of arrival rates $\mu$ such that $(G, \mu)$ is stabilisable, is open.
	
	Indeed, fix $\mu$ such that $(G, \mu)$ satisfies Condition \eqref{eq:gencond}. For each support $U$, let $\nu^U$ and $\gamma^U$ be an assignment rate and a reassignment rate satisfying \eqref{eq:gencond}. Let
	\[
		\varepsilon_G = \min_{U} \min_{(u, e)\in G_U} \left|\frac{1}{e_u} \left[\nu^U_{v, e} + \Gamma^U_{v, e}\right] - \frac{1}{|e|}\left[\nu^U_e + \Gamma^U_e\right]\right| > 0.
	\]
	
	Moreover, for each support $U$ and each vertex $u$, choose an arbitrary hyperedge $e^U_u$ containing $u$ with $\nu^U_{u, e} > 0$, which must exist as $\sum_{e \ni u} \nu_{u, e} = \mu_u > 0$. Denote $\varepsilon^U_\nu = \min_{u} \nu_{u, e^U_u}$, and $\varepsilon_\nu = \min_U \varepsilon^U_\nu$.
	
	Now consider a given $\eta \in \mathbb{R}_{\geq 0}^V$ such that $\|\eta - \mu\|_1 = \max_{u \in V} |\eta_v - \mu_v| \leq \min \left(\varepsilon_\nu, \frac{\varepsilon_G}{3}\right)$. For a given support $U$, let $\nu = \nu^U$. For each $(u, e) \in G$, let
	\[
	\nu'_{u, e} = 
	\begin{cases}
		\nu_{u, e} + \eta_ u- \mu_u & \text{ if } e = e_u \\
		\nu_{u, e} & \text{ otherwise}
	\end{cases},
	\]
	we have that $\nu' \in \mathbb{R}_{\geq 0}^G$ is an assignment rate, as for each $u \in V$, we have $\sum_{e \ni u} \nu'_{u, e} = \eta_u - \mu_u + \sum_{e \ni u} \nu_{u, e} = \eta_u$. 
	
	Finally, for all $(u, e) \in G$, we have $|\nu_{u, e} - \nu'_{u, e}| \leq \|\eta - \mu\|_1 \leq \frac{\varepsilon_G}{3}$, implying
	\[
		|\nu'_e - \nu_e| = \left| \sum_{u \in e} \nu'_{u, e} - \nu_{u, e}\right| \leq \sum_{u \in e} |\nu'_{u, e} - \nu_{u, e}| \leq |e| \frac{\varepsilon_G}{3}.
	\]
	In particular, this implies $\frac{\nu_e}{|e|} \leq \frac{\nu'_e}{|e|} + \frac{\varepsilon_G}{3}$, thus for each $(v, m) \in U$, we have
	\begin{align*}
		\nu'_{v, m} + \Gamma_{v, m}
		& \leq \frac{\varepsilon_G}{3} + \nu_{v, m} + \Gamma_{v, m} \\
		& \leq \frac{\varepsilon_G}{3} - \varepsilon_G + \frac{1}{|e|} \left[ \nu_e + \Gamma_e\right] \\
		& \leq 2\frac{\varepsilon_G}{3} - \varepsilon_G + \frac{1}{|e|} \left[ \nu'_e + \Gamma_e\right] \\
		& \leq -\frac{\varepsilon_G}{3} + \frac{1}{|e|} \left[ \nu'_e + \Gamma_e\right] \\
	\end{align*}
	Therefore, $\nu'$ and $\gamma$ satisfy \eqref{eq:gencond} for $U$ and $\eta$. Since this argument holds for all $U$, $(G, \eta)$ satisfies Condition \eqref{eq:gencond}, hence $(G, \eta)$ is stabilisable.
	
	Going back to conservation equation, since $(G, \mu)$ and $(G, \eta)$ are both stabilisable, there exist vectors $\lambda_\mu$ and $\lambda_\eta$ such that $G_M \lambda_\mu = \mu$ and $G_M \lambda_\eta = \eta_V$. This implies 
	\[
		G_M (\lambda_\eta - \lambda_\mu) = \eta_V - \mu
	\]
	but since this is true for all $\|\eta - \mu\|_1 \leq \frac{\varepsilon_G}{3}$, we conclude that $G_M$ is surjective.
	
	Next, $G_M \lambda = \mu$ having a \textit{non-negative} solution $\lambda_0 \in \mathbb{R}^M_{\geq 0}$ is simply the conservation equation which holds when $(G, \mu)$ is stabilisable. However, to construct a \textit{positive} solution $\lambda \in \mathbb{R}^M_{> 0}$, we need an extra argument.
	
	See $G$ now as a subset of $V \times E$ and write $|G| = |\{(u, e) \mid (u, e) \in G\}|$. Let $0 < \varepsilon < \frac{1}{|G|} \min \left(\varepsilon_\nu, \frac{\varepsilon_G}{3}\right)$, and $\mu_\varepsilon \in \mathbb{R}^V$ be defined by $(\mu_\varepsilon)_u = \mu_u - \sum_{e \ni u} \varepsilon$. We have that $\| \mu_\varepsilon - \mu\|_1 = \varepsilon |G| < \min \left(\varepsilon_\nu, \frac{\varepsilon_G}{3}\right)$ by construction, so by the argument above, we have $(G, \mu_\varepsilon)$ is stabilisable. Thus, it satisfies the conservation equation, meaning there exists $\lambda_\varepsilon \in \mathbb{R}^E_{\geq 0}$ such that $G \lambda_\varepsilon = \mu_\varepsilon$.

	Let $\lambda = \lambda_\varepsilon + \varepsilon \mathbbm{1}$, then by construction, we have $G\lambda = \mu$, and moreover, for all $e \in E$, we also have $\lambda_e = \lambda_\varepsilon + \varepsilon \geq \varepsilon > 0$, as desired.
\end{proof}

\subsection{Sufficiency of \eqref{eq:inccond}}

Recall that we have remarked in Subsection \ref{subsection: contributions and organisation}, that the framework set out for the unweighted case does not directly generalised to the weighted case precisely because of the negative weights. In particular, this condition becomes crucial for the step where we want to show that Condition \eqref{eq:inccond} implies Condition \eqref{eq:onlcond}. But now as we have positive weights, this is no longer a problem.

\begin{lemma}
	\label{lemma: sufficiency of inccond}
	Condition \eqref{eq:inccond} implies Condition \eqref{eq:onlcond}. In particular, $(G, \mu)$ is stabilisable.
\end{lemma}
\begin{proof}
	As in the unweighted case, our starting point is a solution $\lambda \in \R^E_{> 0}$ to the equation $G\lambda = \mu$, which gives us an assignment rate for which the inequality constraint in \eqref{eq:onlcond} is an equality. To tip it to be an inequality, we add a perturbation $\zeta$; this however causes the conservation constraint to fail, so to correct this, we add another term of perturbation, whose construction uses the surjectivity of $G$.
	
	Let $U$ be a support. For some $\varepsilon > 0$ to be defined later, let
	\[
	\zeta_{u, e} = \begin{cases}
		\varepsilon & \text{ if } (u, e) \in U \\
		0 & \text{ otherwise}
	\end{cases}.
	\]
	This induces a vector $\eta \in \mathbb{R}^V$ by $\eta_u = \sum_{e \ni u} \zeta_{u, e}$. Since $G$ is surjective, there exists $\xi \in \mathbb{R}^E$ such that $G \xi = \eta$, or in other words, for all $u \in V$, $\eta_u = \sum_{e \ni u} e_u \xi_e$.
	
	Now for $(u, e) \in G$, let $\nu_{u, e} = e_u \left(\lambda_e - \zeta_{u, e} + \xi_e\right)$, we will show that for suitably chosen $\varepsilon$, $\nu$ is an assignment rate satisfying \eqref{eq:onlcond} for $U$.
	\begin{itemize}
		\item By construction, for all $u \in V$, we have
		\[
		\begin{aligned}
			\sum_{e \ni u} \nu_{u, e}
			& = \sum_{e \ni u} e_u \left(\lambda_e + \zeta_{u, e} - \xi_e\right) \\
			& = \sum_{e \ni u} e_u\lambda_e + \sum_{e \ni u} e_u\zeta_{u, e} - \sum_{e \ni u} e_u\xi_e\\
			& = \mu_u + \eta_u - \eta_u = \mu_u;
		\end{aligned} 
		\]
		\item Moreover, we have $-\zeta_{u, e} < -\frac{1}{|e|} \zeta_e$ where $\zeta_e = \sum_{v \in e} e_v\zeta_{v, e}$. This is because of the definition of support, where for all $e$, there exists $v' \in e$ such that $(v', e) \not\in U$, so $\zeta_e \leq (|e|-\min_{v \in e} e_v) \delta$.
		
		In particular, this implies
		\begin{align*}
			\frac{\nu_{u, e}}{e_u}
			& = \lambda_e - \zeta_{u, e} + \xi_e = -\zeta_{u, e} + \frac{1}{|e|} \sum_{v \in e} e_v \left(\lambda_e + \xi_e\right) \\
			& < -\frac{1}{|e|} \sum_{v \in e} e_v \zeta_{v, e} + \frac{1}{|e|} \sum_{v \in e} e_v \left(\lambda_e + \xi_e\right) = \frac{1}{|e|} \sum_{u \in e} e_v (\lambda_e - \zeta_{v, e} + \xi_e) \\
			& = \frac{1}{|e|} \sum_{v \in e} \nu_{v, e} = \frac{\nu_e}{|e|};
		\end{align*}
		\item Finally, note that $\|\eta\|_1 = |U|\varepsilon$. Since $A$ is a linear map, it is continuous, hence, for $\varepsilon > 0$ small enough, we have that $\|\xi\|_\infty \leq -\varepsilon + \min_{(u, e) \in G} \lambda_e$. This implies that for all $(u, e) \in G$.
		\[
			\nu_{u, e} = \lambda_e - \zeta_{u, e} + \xi_e \geq \lambda_e - \varepsilon - \|\xi\|_\infty \geq 0.
		\]
	\end{itemize}
	Thus $\nu$ satisfies Condition \eqref{eq:onlcond} for the support $U$. Since this holds for all supports $U$, $(G, \mu)$ satisfies Condition \eqref{eq:onlcond}.
\end{proof}

Combining the two lemmata with the fact that \eqref{eq:onlcond} is a special case of \eqref{eq:gencond}, we get that 
\begin{theorem}
	\label{theorem: necessity and sufficiency of inccond}
	Let $G \in \R_{\geq 0}^{V\times E}$. Condition \eqref{eq:onlcond}, \eqref{eq:gencond}, and \eqref{eq:inccond} are equivalent.
	
	In particular, the stationary size-based randomised online assignment policy constructed in the proof of Lemma \ref{lemma: sufficiency of onlcond} and the online assignment version of $L_+$-MaxWeight admits maximum stability region.
\end{theorem}

\begin{remark}
	It is clear that $L_+$-MaxWeight is implementable in polynomial time, but the question of implementing the online assignment policy defined in the proof of Lemma \ref{lemma: necessity of onlcond} has not been addressed.
	
	One can see that given a support, the problem of finding some assignment rate optimising the drift of $L_+$ is essentially a linear programming problem, which is solvable in weakly polynomial time. However, this can be improved: from a solution $\lambda$ to the conservation equation $G \lambda = \mu$, it suffices to identify $\varepsilon > 0$ sufficiently small, at which point the proof of Lemma \ref{lemma: sufficiency of inccond} gives an explicit construction of a desired assignment rate for a given support.
	
	Morever, as this construction works for all $\varepsilon > 0$ sufficiently small, we in fact obtain a family of Markovian, size-based online assignment policies admitting maximum stability region. 
\end{remark}
	\section{Generalised matchings}
\label{section: generalised matchings}

Now that we have established necessary and sufficient criteria for stability in the case of non-negative weights, let us tackle the case of general weights. To this, we introduce the notion of generalised matching types.

\subsection{Motivation and definition}

As mentioned in Subsection \ref{subsection: contributions and organisation}, the main difficulty in adapting the framework of the unweighted case to the general-weight case is the fact that there is a key step in the proofs \cite[Lemma 5.2]{Nguyen2026} where we require that $e_u > 0$ for all $(u, e) \in G$.

To see where this gap comes from, let us look at what happens: a type-$e$ matching can generate new items, say of class $v$, but these items cannot stay in the buffer indefinitely if the system is stable. As such, they must participate in other matchings, and eventually leave the system.

This is where the previous framework shows its inadequacy: in the classical setting, we have shown that there is an online assignment policy that is maximally stable \cite[Remark after Theorem 5.1]{Nguyen2026}. However, such policy fixes the matching type of \textit{all} items involved in the matching. In particular, it means that if we were to use the same policy for general weighted case above, the newly created class-$v$ items would be assigned to type-$e$ matchings by default, to which they could not participate because $e_v < 0$ by definition, and they could not be matched by any other matching types either. Hence, they would accumulate, and the system would be unstable. This shows that the weakness is fundamental, of which the would-be gap in the proofs is only a manifestation.

As such, a more general framework is necessary. To come back to the example above, the insight is that these newly created items will leave the buffer, meaning that they participate in some other matchings. Eventually, the system is stable, so \textit{all} newly created items will leave the buffer, via some combinations of matchings.

To make this intuition rigorous, we define generalised matching types, which are essentially $\N$-linear combinations of matching types.

\begin{definition}
	A generalised matching $m$ is a sequence $(m_{v, n})_{v \in V, n \in \N}$ in $K_V^\N$ such that for some $\N$-linear combination of hyperedges $\sum_{e \in E} m_e e$ and for all $v \in V$, we have $\sum_{n = 0}^\infty m_{v, n} = \sum_{e \in E} m_e e_v$.
	
	When the exact choice of item is not required, we ignore the index $n$, in which case we say that a generalised \emph{matching type} $m$ is a $\N$-linear combination of hyperedges $\sum_{e \in E} m_e e$. We write $m_v = \sum_{e \in E} m_e e_v$, and $|m| = \sum_{v \in V} m_v$.
	
	Moreover, let $B = (b_{v, n}) \in K_V^\N$ be a buffer state, then a generalised matching $m$ is admissible if for all $v \in V$ and $n \in \N$, we have $m_{v, n} \leq b_{v, n}$. If we choose to realise this matching, then the resulting buffer state $C = (c_{v, n})$ is given by $c_{v, n} = b_{v, n} - m_{v, n}$.
\end{definition}

From now on, for the sake of brevity, we omit the adjective ``generalised'' whenever the context is clear. Now we say that a generalised matching is stabilising if it does not increase the content of the buffer for all vertices.

\begin{definition}
	A generalised matching type $m = (m_{v, n})_{v \in V, n \in N}$ is said to be stabilising if for all $v \in V$, we have $m_v \geq 0$. We say that $m$ uses class $v$, and write $v \in m$, when $m_v > 0$.
\end{definition}

\subsection{Existence}

The first question is whether for every vertex $v$, there are stabilising generalised matchings which use class-$v$ items, meaning $m_v > 0$. The answer is no: consider the closed system with some initial items but there are no external arriving items, then the system is obviously stabilisable, but there are no stabilising matchings that actually use any items, since the total number of items in the system is fixed.

However, this is an exception. We rather need stabilising matchings $m$ using class-$v$ items when there can (and will) be external arriving class-$v$ items, so the question now is whether for all such vertex $v$, there are stabilising generalised matchings which use class-$v$ items. Fortunately, the answer is yes.

\begin{lemma}
	\label{lemma: existence of generalised matching types}
	Suppose $(G, \mu)$ is stabilisable, then for all $v \in V$ such that $\mu_v > 0$, there exist stabilising generalised matchings $m$ such that $m_v > 0$.
\end{lemma}
\begin{proof}
	Let $V' = \{ v \mid \mu_v > 0\}$. As $C$ is compact, it is bounded, so we can choose $M$ such that $M > x_v$ for all $x \in C$ and $v \in V'$.
	
	By definition, with probability $1$, for some $T$ large enough, we have that for all $v \in V'$, the total number of class-$v$ items arriving up to time $T$ is at least $M$. But if $(G, \mu)$ is stabilisable, we have that $\E[\tau \mid \mathcal{H}_T] < \infty$, where
	\[
		\tau = \min \{t \geq 0 \mid Z(t + T) \in C \}.
	\]
	So at time $T + \tau$, we have that $Z_v(T + \tau) < M$ by definition of $M$; we also have that $Z_v(T) \geq M$. This means that between $T$ and $T + \tau$, some simple matchings occurs. Combining all such matchings as a generalised matching $m$, we have that $m_v = Z_v(T) - Z_v(T + \tau) > 0$ for all $v \in V'$, as desired.
\end{proof}

\subsection{Generators}
\label{subsection: generators}

We can now consider the set of stabilising generalised matchings $M = \{ m \mid m_v \geq 0\}$. It is tempting to replace $E$ with $M$, and in particular, we can define a new hypergraph $G_M = (V, E_M)$ where for every stabilising matching $m \in M$, we associate a weighted hyperedge $e_m = \{v \mid m_v > 0\}$. It is clear that $(G, \mu)$ is stabilisable if and only if $(G_M, \mu)$ is also stabilisable, so this would bring us back to the classical setting, and indeed this is what we will do.

But before that, there is one subtlety: unlike $E$, $M$ is \textit{not} finite, for any integer multiple of a stabilising matching is also stabilising. In this case, we would end up with an infinite hypergraph, which is not the classical setting and for which the previous framework might not be suitable.

To this problem, we first investigate the structure of $M$. If we see each hyperedge $e$ as a column vector in the matching matrix, then by definition, one can also identify $M$ with the set
\[
	\left\{x \in \N^E \Bigg\vert \sum_{e \in E} x_e e \geq 0\right\} = \left\{x \in \R^E \Bigg\vert \sum_{e \in E} x_e e \geq 0\right\} \cap \R^E_{\geq 0} \cap \Z^E.
\]

The first two sets, $\left\{x \mid \sum_{e \in E} x_e e \geq 0\right\}$ and $\R^E_{\geq 0}$, are convex cones, thus so is their intersection. We then see that $M$ is the set of integer points in a convex cone.

If $G$ has rational weights, then this cone is a rational cone. By Gordan's lemma, $M$ is finitely generated. Since we may consider the stabilising matchings up to multiplicities, we can take a finite set of generators $E'$ for $M$ and define the hypergraph $G' = (G, E')$, for which the previous framework applies.

However, if $G$ has irrational weights, this remedy is not possible, and there are cases where $M$ is not finitely generated. The hypergraph $G_M$ is now truly infinite, to which the framework established in the previous Sections does not apply. A way to overcome this obstacle is the topic of next Section. 
	\section{General-weight matchings}
\label{section: general weights}

In the previous Section, we have reduced the stability of $(G, \mu)$ to that of $(G_M, \mu)$. Before, when $G$ has non-negative weights, we have assumed that $\mu_v > 0$ for all $v \in V$, with the reason being that if $\mu_v = 0$, then there is (almost surely) no class-$v$ items arriving. A matching type $e$ thus either does not use class-$v$ items, or it does, in which case eventually it will no longer be feasible and thus can be safely ignored.

When $G$ is of arbitrary weights, it may occur that even when $\mu_v = 0$, whilst some matching types \emph{use} class-$v$ items, other matching types \emph{generate} them. In view of a generalised stabilising matching type $m$, by the same reasoning as in the previous paragraph, we must have $m_v = 0$, so the number of class-$v$ items in the system is essentially conserved.

However, it may also happen that while $m_v = 0$, some constituent matching types $e$ use class-$v$ items, which may not be feasible (and thus also hinder $m$) if there is not enough of them in the first place. We assume hereinafter that this is not the case: namely, for all $v \in V$ such that $\mu_v = 0$, we have $Z_v(0)$ sufficiently large so that all stabilising generalised matching types are not hindered by the lack of class-$v$ items.

As we have remarked, $M$ is infinite, and a priori we might need to use an infinite subset of $M$ in order to stabilise $(G, \mu)$. However, fortunately, it turns out that we do not have to.

Consider the set $C_V$ of arrival rate $(\mu_v)_{v \in V}$ such that $(G, \mu)$ is stabilisable. Within $C_V$, we consider a subset $F_V$, the set of arrival rates $\mu$ such that $(G, \mu)$ is stabilisable using finitely many stabilising generalised matching types.

The first fact about $F_V$, however, is rather intriguing.
\begin{lemma}
	\label{lemma: density of F_V}
	$F_V$ is dense in $C_V$.
\end{lemma}
\begin{proof}
	Let us define $N$ as the formal linear map sending $x \in \R^M$ to the vector $\left(\sum_{m \in M} x_m m_e\right)_{e \in E}$ where $m_e$'s are the coefficients in a representation $m = \sum_{e \in E} m_e e$. It is entirely possible that this representation is not unique, in which case we pick one arbitrarily. Let us also write $G_M = G \circ N : \R^M \to \R^V$.
	
	Now let $\mu \in C_V \setminus F_V$, we have that $(G, \mu)$ is stabilisable, so in particular one can talk about $\lambda_m$ the average number of type-$m$ matchings during a given unit of time. This constitutes a vector $\lambda$, which b the conservation equation, satisfies $G_M \lambda = \mu$. In particular, this implies $\sum_{m \in M} \lambda_m |m| = \|\mu\|_1$.
	
	Let $\varepsilon > 0$, and take $M'$ a subset of $M$ such that $\sum_{m \in M\setminus M'} \lambda_m |m| \leq \varepsilon$. Let $\mu'_{\varepsilon, v} = \sum_{m \in M'} \sum_{e \in E} \lambda_m m_e e_v$, we have that 
	\[
		\mu_v - \mu'_{\varepsilon, v}  = \sum_{m \in M\setminus M'} \sum_{e \in E} \lambda_m m_e e_v \leq \max_{e \ni v} e_v \sum_{m \in M\setminus M'} \sum_{e \in E} \lambda_m m_e \leq \varepsilon\max_{e \ni v} e_v,
	\]
	so in particular, 
	\[
		\|\mu - \mu'_\varepsilon\|_1 \leq \varepsilon \sum_{v \in V} \max_{e \ni v} e_v,
	\]
	uniformly. As such, letting $\varepsilon$ tend to $0$, we have $\|\mu - \mu'\|_1 \rightarrow 0$.
	
	For $\left(G_{M'}, \mu'\right)$, the conservation is satisfies by construction, and we can always take $M'$ to contain a basis of $\R^E$ which is finite dimensional, so that the map $G_{M'}$ is surjective. Hence, $\left(G_{M'}, \mu'\right)$, and thus $\left(G, \mu'\right)$, is stabilisable using finitely many stabilising matching types, which means $\mu'$ is in $F_V$.
\end{proof}

The second fact about $F_N$ is no less non-trivial.

\begin{lemma}
	$F_V$ is closed.
\end{lemma}
\begin{proof}
	Let $(\mu_n)_{n \in \N}$ be a sequence of arrival rates in $F_V$ converging to an arrival rate $\mu$ in $C_V$. For each $n$, $(G, \mu_n)$ is stabilisable using a finite set $M_n$ of stabilising matching types. Without loss of generality, we may assume that $(M_n)_{n \in \N}$ is a non-decreasing sequence.
	
	In fact, more can be said: once we know that for some $n$, $(G, \mu_n)$ is stabilisable with a finite subset $M_n \subseteq M$, let $\lambda_{n, m}$ be the average number of type-$m$ matchings during a given unit of time, we have that the vector $\lambda_n = \left(\lambda_{m, m}\right)_{m \in M_n}$ satisfies $G_{M_n} \lambda_n = \mu_n$. Let $\omega_n = N_n \lambda_n$ where $N_n$ is the formal linear map we defined at the beginning of Lemma \ref{lemma: density of F_V} but for $M_n$, we have that $\omega$ is a $\R_{\geq 0}$-linear combination of columns of $N_n$, which are in $\R^E$. By Carath\'{e}odory's theorem, there exists a subset $M'_n$ of $M_n$ of size at most $|E|$, such that $\omega$ can be written as a $\R_{\geq 0}$-linear combination the corresponding columns of $N'$.
	
	On the other hand, by Theorem \ref{theorem: necessity and sufficiency of inccond}, we have that the matrix $G_{M_n} = G \circ N_n$ is full-rank. As $G$ is map from $\R^E$ to $\R^V$, this implies $|E| \geq |V|$, and similarly, the surjectivity of $N_n$ implies $|M_n| \geq |E|$. From the subset $M'_n$ in the previous paragraph, we may extract a minimal generating set, then extend it to a subset of size $|E|$ such that the corresponding columns in $N_n$ form a basis of $\R^E$. The corresponding formal map $N'_n$ is now bijective, and the matrix $G_{M'_n}$ is thus full-rank.
	
	When we write $\omega_n$ as a $\R_{\geq 0}$-linear combination of columns of $N'_n$ corresponding to $M'_n$, it is possible that some coefficients are zero, but it suffices to repeat the argument at the end of Lemma \ref{lemma: necessity of incccond} to obtain strictly positive coefficients.
	
	By Theorem \ref{theorem: necessity and sufficiency of inccond}, $(G_{M'_n}, \mu)$ is stabilisable. Finally, since $(M_n)_{n \in \N}$ is a non-decreasing sequence, we have $M'_n \subset M_{n'}$ for all $n' \geq n$, so in fact this construction is well-defined at the limit, and we may construct $\omega_n$ so that as $n$ tends to infinity, it converges to a limit $\omega$.
	
	Thus in sum, if $(G, \mu)$ is stabilisable, then it also is with the finite set $M'_n$ of stabilising matching types, meaning $\mu \in F_V$.
\end{proof}

\begin{remark}
	It is also clear from the proof that $|M'_n| = |E|$ and this cannot be improved.
\end{remark}

As a corollary of both lemmata, we obtain that $F_V = C_V$, and for all arrival rates $\mu$, $(G, \mu)$ is stabilisable using finitely many stabilising generalised matchings.

\begin{theorem}
	$(G, \mu)$ is stabilisable if and only if it is stabilisable using a set $M'$ of finitely many matching types. When this is the case, one may need exactly $|E|$ stabilising generalised matching types and this is optimal.
	
	In view of Theorem \ref{theorem: necessity and sufficiency of inccond}, this is equivalent to the existence of a matrix $N \in \N^{E \times E}$ such that the map $G_N = G\circ N$ is surjective and the equation $G_N \lambda = \mu$ admits a positive solution $\lambda \in \R_{> 0}^E$.
	
	With respect to $(G_{M'}, \mu)$, both $L_+$-MaxWeight and the family of Markovian, size-based, online assignment polices constructed in the proof of Lemma \ref{lemma: necessity of onlcond} admit maximum stability region.
\end{theorem}
	\section{Continuous-time matchings}
\label{section: continous-time matchings}

In the previous Sections, we have developed a framework to analyse the stability of discrete-time matching models. The goal of this Section is to generalise what we have done to continuous time.

We may keep the matching hypergraphs and its derived notions as defined in Section \ref{section: preliminaries} and Section \ref{section: generalised matchings}, but the arrival process must be redefined. Indeed, a crucial difference with respect to the previous case is that now we no longer have clearly the ``before''-buffer state $X(t-1)$ and the ``after''-buffer state $X(t)$. As such, our first work is to define the arrival processes in this new setting.

\subsection{Continuous arrivals}

In the usual setting, we have that the arrival at time $t$ is i.i.d. sampled from a fixed distribution, or in other words, if we denote by $S(t) = (S_v(t))_{v \in V}$ the cumulative arrival process, then its increment $S(t+1) - S(t)$ is independent and stationary. Moreover, by definition, we have that $S(t)$ is a non-decreasing and non-negative process.

We use these two properties to generalise the arrival process to the continuous-time setting. More precisely, let $S(t) = (S_v(t))_{v \in V}$ be the cumulative arrival process, and in particular $S_v(t)$ is the quantity of class-$v$ items arrived up to time $t$, we require that $S$ be a multivariate subordinator.

We can write $S(t)$ in integral form. Let $\Pi$ be the unique L\'{e}vy measure associated with $S$ given by L\'{e}vy-Khintchine representation and $N(\cdot)$ be the corresponding Poisson random measure of intensity $\dd s \Pi(\dd k)$, we have
\[
	S(t) = f\cdot t + \int_{(0, t] \times \R_{> 0}^V} k N(\dd s, \dd k).
\]

Writing $S(t) = (S_v(t))_{v \in V}$, for a given $v \in V$, the second moment of $S_v(t)$ is given by
\[
	\E\left[S_v(t)^2\right]  = t \int_{\R_{\geq 0}^V \setminus \{0\}} k_v^2 \Pi(\dd k) + t^2 (\mu_v)^2,
\]
where $\mu_v = f_v + \int_{\R_{\geq 0}^V \setminus \{0\}} k_v \Pi(\dd k)$. As in the batch arrival case \cite{Nguyen2026}, we assume that $\E[S_v(t)^2]$ is finite for all $t$, which implies that $\int_{\R_{\geq 0}^V \setminus \{0\}} k_v^2 \Pi(\dd k)$ is finite. On the other hand, note that since $\Pi$ is a L\'{e}vy measure, we have
\[
	\int_{\|k\| \leq 1} k_v \Pi(\dd k) \leq \int_{\R_{\geq 0}^V \setminus \{0\}} (1 \wedge k_v^2) \Pi(\dd k) < \infty,
\]
so in particular, if $\int_{\R_{\geq 0}^V \setminus \{0\}} k_v^2 \Pi(\dd k)$ is finite, we have
\[
	\int_{\|k\| > 1} k_v \Pi(\dd k) \leq \int_{\|k\| > 1} k_v^2 \Pi(\dd k) < \infty,
\]
implying $\E\left[S_v(t)^2\right] < \infty$ for all $t$. As such, it would that we assume for all $v \in V$,
\[
	\int_{\R_{\geq 0}^V \setminus \{0\}} k_v^2 \Pi(\dd k) < \infty,
\]
but we will in fact assume that
\[
	\|\mu\|_2^2 = \int_{\R_{\geq 0}^V \setminus \{0\}} \|k\|^2 \Pi(\dd k) < \infty,
\]
which is the same claim when $V$ is finite.

Furthermore, let us write $\mu_v^c = f_v$ and $\mu_v^b = \int_{\R_{\geq 0}^V \setminus \{0\}} k_v \Pi(\dd k)$. Note that the mean of $S_v(t)$ is precisely
\[
	\E[S_v(t)] = t \left(f_v + \int_{\R_{\geq 0}^V \setminus \{0\}} k_v \Pi(\dd k)\right) = t(\mu_v^c + \mu_v^b) = t \mu_v,
\]
which allows us to interpret $\mu_v$ as the average arrival rate of class-$v$ items, and our assumption allows us to deduce $\mu_v < \infty$, which is a desirable property. In particular, $\mu_v^c$ and $\mu_v^b$ are the average number of class-$v$ items arriving as a continuous flow and as batches, respectively. Denote
\[
	\|\mu\|_1 = \sum_{v \in V} \mu_v = \int_{\R_{> 0}^V} \|k\|_1 \Pi(\dd k).
\]

To see the link between this arrival process and the discrete-event case, we write the term $S^b(t)$ as a countable sum of compound Poisson processes. For $i \geq 1$, let
\[
E_i = \left\{k \in \R_{\geq 0}^V : \frac{1}{n+1} < \|k\| \leq \frac{1}{n}\right\},
\]
and $E_0 = \{k \in \R_{\geq 0}^V : \|k \| > 1\}$, then we have $\R_{\geq 0}^d \setminus \{0\} = \bigcup_{n = 0}^\infty E_n$. As $\Pi$ is a L\'{e}vy measure, we have that 
\[
\int_{\R_{\geq 0}^V \setminus \{0\}} (1 \wedge \|k\|) \Pi(\dd k) < \infty,
\]
and in particular, we have $n_i = \Pi(E_n) < \infty$. Then, using L\'{e}vy-It\^{o} decomposition, we can write
\[
S(t) = f \cdot t + \sum_{i = 0}^\infty \sum_{j = 0}^{N_i(t)} \xi_{i, j}
\]
for some $f \in \R_{\geq 0}^V$, $N_i(t) \sim \text{Poisson} (n_i t)$ and $\xi_{i, j}$ are drawn i.i.d. from the distribution $\P(\xi_{i, j} \in \cdot) = \frac{1}{n_i} \Pi(\cdot \cap E_i)$, which is the (normalised) distribution $\Pi$ restricted to the set $E_i$.

This sum has two terms: the linear term $S^c(t) = f \cdot t$, which corresponds to the \textit{continuous} inflow of raw material, and the pure-jump term, $S^b(t) = \sum_{i = 0}^\infty \sum_{j = 0}^{N_i(t)} \xi_{i, j}$, which corresponds to the \textit{batch} arrivals. In particular, let $T_{i, n}$ be the $n$\textsuperscript{th} arrival time of the Poisson process $N_i$, then at time $T_{i, n}$, there is a batch of size $\xi_{i, j} \in \R_{\geq 0}^V$ arriving.

From here, we recognise the discrete-time model as a special case, where $f = 0$ and the pure-jump term is simply a Poisson process instead of a countable sum of compound Poisson processes. As such, this definition of arrival process generalises the discrete-event case.

\subsection{Matching models}

Together with the matching hypergraph $G$, the pair $(G, S)$ is called a matching model. We equip each vertex $v \in V$ with a queue $Z_v(t)$ taking value in $K_v$; the whole queue-length process $Z(t)$ takes values in $K_V = \prod_{v \in V} K_v$.

The main reason for requiring $Z(t) \in K_V$ instead of $\R_{\geq 0}^V$ comes from the discrete setting where $f = 0$. For a given class $v \in $, if $f_v > 0$ then it means that items of this class will arrive as a continuous flow, which suggests that they are infinitely divisible. However, if $f_v = 0$, then the items only arrive in batches, which may or may not be divisible.

\begin{remark}
	Readers might notice that unlike the discrete-event case, we have ignored the time of arrival (see. Section \ref{section: preliminaries}). Indeed, it is possible to specify the arrival time: recall that in the discrete setting, we have defined the buffer $\hat{Z}_v(t)$ as a sequence $(\hat{z}^t_{v, n})_{n \in \N} \in K_V^\N$ where $\hat{z}^t_{v, n}$ is the ``number'' of items arrived at time $n$ and still in the buffer at time $t$. Replacing the index $n \in \N$ with $n \in \R$ and writing
	\[
		\hat{Z}_{v, n}(t) = f_v + \sum_{i = 0}^\infty \sum_{j = 0}^\infty (\xi_{i, j})_v \mathbbm{1}_{t = T_{i, j}}
	\]
	where $T_{i, j}$ is the $j$\textsuperscript{th} arrival time of the Poisson process $N_i$, we still have that
	\[
		\int_\R \hat{Z}_{v, n}(t) \dd n = Z_v(t),
	\]
	similar to the identity $\sum_{n = 0}^\infty \hat{Z}_{v, n}(t) = Z_v(t)$ in the discrete-event case.
	
	The reason we do not do this is two-fold. On the one hand, this considerably complicates our (already rather heavy) representation, which later leads to further complication in analysis. For instance, for a given matching, a matching policy must decide how many items it takes from $\hat{Z}_{v, n}(t)$, which potentially incurs measurability issue.
	
	On the other hand, as we will see, the notion of stability is defined using $Z$ instead of $\hat{Z}$, and much analysis is based on $Z$ instead of $\hat{Z}$. In particular, as we will see and may expected from the discrete setting, whenever the matching model $(G, S)$ is stabilisable, there is a stabilising matching policy using only $Z$, so this additional layer of rigour brings no benefits.
\end{remark}

\subsection{Matching policies}

Informally, a realisation of the model can be understood as follow: we start with some initial buffer $Z(0) \in K_V$. Then, at time $t$, some items arrive potentially as batches and as a tiny amount of flow, which together is denote by $\dd S(t)$, to the ``before'' buffer $Z(t^-)$. Then, some $\dd Q(t)$ matching occurs, and the resulting buffer is $Z(t)$.

With this intuition in mind, we define a cumulative matching process $Q(t) = (Q_v(t))_{v \in V}$ with $Q(0) = 0$ as a non-decreasing process taking value in $K_V$ and satisfying $Q(t) \leq S(t)$ for all $t$, meaning that it does not use more items than what has arrived at the time. Moreover, as a matching policy may only use at time $t$ the information available to it at the moment with possibly some independent source of randomness. In particular, let $U = (U(t))_{t \geq 0}$ be a sequence of independent random variables, and $\mathcal{H}_t = \sigma\left(\left\{Z(0), S^c(t'), S^b(t), U(t') \mid t' \leq t\right\}\right)$ be the Borel $\sigma$-algebra generated by the arrivals $(S(t'))_{t' \leq t}$ up to time $t$ and the randomness available thus far $(U(t'))_{t' \leq t}$, then we require that $Q$ be adapted to $\mathcal{H}$.

The resulting buffer at time $t$ can be written as
\[
	Z(t) = Z(0) + S(t) - Q(t),
\]
and we identify our usual notion of matching policy $\Phi$ with the matching process $Q(t)$. This is the description of a matching policy in generality, but as in the discrete setting, we may have relaxations such as Markovian, stationary, and size-based polices.

\begin{definition}
	\hfill
	\begin{itemize}
		\item A matching policy $\Phi$ is said to be size-based if instead of containing $S^c(t)$ and $S^b(t)$, we require $\mathbb{H}_t$ to contain only $S(t) = S^c(t) + S^b(t) \in \mathcal{H}_t$. In particular, it ignores the individual batches, which is specified by the batch arrival process $S^b(t)$, and only takes into account the total size $S(t)$.
		\item A matching policy $\Phi$ is said to be Markovian if for all $0 \leq s \leq t$, its behaviour in the future $Q(t)$ depends only on the current buffer state $X(s)$ and the arrivals in-between $(S(t'))_{s \leq t' \leq t}$, with possibly some randomisation. Specifically, for all $B \in \mathcal{B}(K_V)$, we have
		\[
		\P\left[Q(t) \in B \mid \mathcal{H}_s, (S(t'))_{s < t' \leq t}\right] = \P\left[Q(t) \in B \mid s, X(s), (S(t'))_{s \leq t' \leq t}\right].
		\]
		If the kernel does not depend on $s$, then we say that the matching policy is stationary.
	\end{itemize} 
\end{definition}

Note that even when $\Phi$ is a Markovian policy, the cumulative matching process $Q(t)$ is not Markovian in general, for $Q(t)$ depends on the arrivals $(S(t'))_{s \leq t' \leq t}$, in addition to $Q(s)$. However, the pair $(Q, S) = (Q(t), S(t))_{t \in \R_{\geq 0}}$ is a Markov process in this case.

Regardless of the policy's properties, the goal is to ensure stability. Following the discrete-event case, we say that $\Phi$ stabilises $(G, S)$ if at any given point, the resulting queue-length process $Z$ is expected to be ``small'' again in finite time. Formally speaking, we require that there exists a compact set $C \subset K_V$ such that for all $T \geq 0$, we have that
\[
	\E\left[\min\{t \geq 0 \mid Z(T + t) \in C\} \mid \mathcal{H}_T\right] < \infty.
\]

\subsection{Periodic reviewing and stability}

It turns out that as complicated as it seems, there is a natural way to reduce the case of continuous time to that of discrete-time, via time aggregation.

To see how, let us come back to the discrete-time models $(G, \mu)$. As we have seen, our stability criteria, Condition \eqref{eq:onlcond}, \eqref{eq:gencond}, and \eqref{eq:inccond}, only relate to $\mu$ via the arrival rates $(\mu_v)_{v \in V}$. If we partition the arrivals into the block of $n$ time steps, meaning we treat the arrivals between $nk$ and $n(k+1)-1$ together as one big batch, effectively we are speeding time up by a factor of $n$, so $\mu$ is scaled up by $n$.

But it also suffices to scale up the original solution $\lambda$ to the conservation equation $G_{M'} \lambda = \mu$ by $n$, and the surjectivity of $G_{M'}$, which depends on $\lambda$ only up to a scaling, remains the same. Per Condition \eqref{eq:inccond}, this new model $(G, n\mu)$ is stabilisable if and only if $(G, \mu)$ is also stabilisable. For non-negative weights, another way to see this is that it suffices to multiple the assignment rates in Condition \eqref{eq:onlcond} (and eventually the reassignment rate \eqref{eq:gencond}) by $n$, so the stability is preserved.

The key idea now is to do the same for continuous time. Namely, let us choose a period $T$. Given an arrival process $S$, we define two auxiliary discrete-time arrival processes $(S_-(t))_{t \in \N}$ and $(S_+(t))_{t \in \N}$. For $S_-$, we define $S_-(k) = S(kT)$. For $S_+$, we define $S_+(k) = S((k+1)T)$. As $S$ is non-decreasing, we have that $S_-\left(\left\lfloor \frac{t}{T}\right\rfloor\right) \leq S(t) \leq S_+\left(\left\lfloor \frac{t}{T}\right\rfloor\right)$, meaning $S_-$ and $S_+$ are an lower bound and an upper bound for $S$, up to an appropriate time scaling.

Informally, what we are doing is to start both $S_-$ and $S_+$ with the same initial buffer $S(0)$. Then, for each period $((k-1) T, kT]$, we consider the arrivals during this period as one batch $\Delta(k) = S(kT) - S\left(\left((k-t)T\right)^-\right)$. For $S_-$, instead of spreading it out, we give the whole batch $\Delta(k)$ at the beginning of the period, and for $S_+$, we give it at the end of the period.

Clearly, this means that if $(G, S_-)$ is stabilisable, then so is $(G, S)$, and thus so is $(G, S_+)$. However, both $(G, S_-)$ and $(G, S_+)$ can be seen as an instance of the same matching model $(G, T\mu)$ (though with different initial buffers). Indeed, as $S$ is a subordinator, the increments $\Delta(k)$ are stationary, and in particular, we have the arrival rate to be exactly $T\mu$. Per the discussion at the beginning of this Subsection, we may ignore the period length $T$, and the stability of $(G, S)$ is precisely equivalent to that of $(G, \mu)$.

\begin{theorem}
	The continuous-time matching model $(G, S)$ is stabilisable if and only if so is the discrete-time $(G, \mu)$. When this is the case, using periodic review of any period, both $L_+$-MaxWeight and the family of Markovian, size-based, online assignment polices constructed in the proof of Lemma \ref{lemma: necessity of onlcond} admit maximum stability region.
\end{theorem}

\begin{remark}
	Note that the fact that the moments of reviewing need not be regular. One may replace $0, T, 2T, \ldots$ with any sequence $0 = t_0 < t_1 < t_2 < \ldots$ where $(t_{k+1} - t_k)_{k \in \N}$ is uniformly bounded.
\end{remark}

	\section{Conclusion}
\label{section: conclusion}

We have defined stochastic matching models with weights and batch arrivals, where the arrival process satisfies some mild technical conditions. We also extended the framework of online assignment and reassignment from the unweighted case to this new setting.

Some differences, such as the fact that $K_V$ needs not be finite and the weights may be negative, prevents the proofs of the unweighted case from carrying over, but we have correctly formulated and proven the necessity and the sufficiency of some (equivalent) criteria for stability, namely Condition \eqref{eq:onlcond}, \eqref{eq:gencond}, and \eqref{eq:inccond}, which generalise those for the unweighted case.

Using the notion of stabilising generalised matching types, we have reduced the general-weighted case to the non-negative weighted case, and show that one may consider exactly $|E|$ stabilising generalised matching types to stabilise the model whenever it is stabilisable, and this number is optimal. As these results depend only on $G$ and the arrival rates $(\mu_u)_{v \in V}$, the same criteria hold for continuous-time matching models.

Amongst the novelties, we constructed a new Lyapunov function $L_+$ and show that $L_+$-MaxWeight policy is maximally stable. On the other hands, the proofs also give a (family of) maximally stable online assignment policies that are efficiently implementable.
	
	\bibliographystyle{ieeetr}
	\bibliography{refs}
\end{document}